\documentclass[11pt, a4paper]{amsart}
\usepackage[utf8]{inputenc}

\usepackage{amsmath, amsfonts, amssymb, amsthm,mathtools,enumerate}
\allowdisplaybreaks
\usepackage{graphicx}
\usepackage{color}
\usepackage[inline]{enumitem}
\usepackage{hyperref}
\usepackage{stmaryrd} %\llbracket \rrbracket
\usepackage{multicol}
\usepackage{tikz-cd}

\usepackage[
    a4paper,
    top=2.5cm,
    bottom=2.5cm,
    left=2.5cm,
    right=2.5cm,
    includeheadfoot,
    headheight=14pt,
    footskip=1cm,
    marginparwidth=3cm,
    marginparsep=0.5cm
]{geometry}
\title{An alternative $\mathbb{Q}$-form of the cyclotomic double shuffle Lie algebra}
\author{Hidekazu Furusho}
\address{\scriptsize Graduate School of Mathematics, Nagoya University, Furo-cho, Chikusa-ku, Nagoya, 464-8602, Japan.}
\email{furusho@math.nagoya-u.ac.jp}
\author{Khalef Yaddaden}
\address{\scriptsize Graduate School of Mathematics, Nagoya University, Furo-cho, Chikusa-ku, Nagoya, 464-8602, Japan.}
\email{khalef.yaddaden.c8@math.nagoya-u.ac.jp}
\date{February 2, 2025}

\newtheorem{thm}{Theorem}[section]
\newtheorem{lem}[thm]{Lemma}
\newtheorem{cor}[thm]{Corollary}
\newtheorem{prop}[thm]{Proposition}

{\theoremstyle{definition} \newtheorem{rem}[thm]{Remark}}
{\theoremstyle{definition} \newtheorem{defn}[thm]{Definition}}
{\theoremstyle{definition} \newtheorem{thm-defn}[thm]{Theorem-Definition}}
{\theoremstyle{definition}  }
{\theoremstyle{definition}  }

\newtheorem{prob}[thm]{Problem}
{\theoremstyle{remark} }

\numberwithin{equation}{section}
\newcommand{\Q}{\mathbb{Q}}

\newcommand{\Z}{\mathbb{Z}}
\newcommand{\Li}{\operatorname{Li}}
\newcommand{\dmr}{\mathfrak{dmr}}
\newcommand{\dmrd}{\mathfrak{dmrd}}

\DeclareFontFamily{U}{wncy}{}
\DeclareFontShape{U}{wncy}{m}{n}{<->wncyr10}{}
\DeclareSymbolFont{mcy}{U}{wncy}{m}{n}
\DeclareMathSymbol{\sh}{\mathord}{mcy}{"78} % use \sh commande for shuffle product

\hypersetup{pdfborder={0 0 0}}

\newlist{caselist}{enumerate}{1}
\setlist[caselist]{label=\textbf{Case \arabic*}, align=left, wide=1pt}

\begin{document}

%%%%%%%%%%%%%%%%%%%%%%%%%%%%%%%%%%%%%%%%%%%%%%%%%%%%%%%%%%%

\begin{abstract}
    We present an alternative $\mathbb{Q}$-form for Racinet's cyclotomic double shuffle Lie algebra, inspired by the double shuffle relations among congruent multiple zeta values studied by Yuan and Zhao. Our main result establishes an invariance characterization theorem, demonstrating how these two $\mathbb{Q}$-forms can be reconstructed from each other under Galois action. 
\end{abstract}

%%%%%%%%%%%%%%%%%%%%%%%%%%%%%%%%%%%%%%%%%%%%%%%%%%%%%%%%%%%

\maketitle

{\footnotesize \tableofcontents}

%%%%%%%%%%%%%%%%%%%%%%%%%%%%%%%%%%%%%%%%%%%%%%%%%%%%%%%%%%%
\setcounter{section}{-1}
\section{Introduction}
Multiple zeta values (MZVs in short) are real numbers defined by the following series
\begin{equation*}
    \zeta(k_1, \dots, k_r) = \sum_{n_1 > \dots > n_r > 0} \frac{1}{n_1^{k_1} \cdots n_r^{k_r}},
\end{equation*}
where $r, k_1, \dots, k_r$ are positive integers with $k_1 \neq 1$ (\cite{Zag94}). These values appear in the studies of many subjects in mathematics and physics such as mixed Tate motives, knot invariant, the KZ-equation and Feynman integrals to name a few (for example, the reader may refer to \cite{B15,BF,Zha16} for a comprehensive presentation).

For a positive integer $N$, one may consider two level $N$ generalizations of multiple zeta values, such that MZV correspond to the $N=1$ case. The first generalization (cf. \cite{Gon98}), called \emph{multiple polylogarithm values at $N$\textsuperscript{th} roots of unity} ($N$-MPVs in short), are complex numbers defined by the following series
\begin{equation}
    \label{NMPV}
    \Li_{(k_1, \dots, k_r)}(\zeta_1, \dots, \zeta_r) = \sum_{n_1 > \dots > n_r > 0} \frac{\zeta_1^{n_1} \cdots \zeta_r^{n_r}}{n_1^{k_1} \cdots n_r^{k_r}},
\end{equation}
where $r, k_1, \dots, k_r$ are positive integers and $\zeta_1, \dots, \zeta_r$ are $N$\textsuperscript{th} roots of unity with $(k_1, \zeta_1) \neq (1, 1)$.
They are also called multiple $L$-values in \cite{AK04}, cyclotomic multiple zeta values in \cite{Ter04} or the colored multiple zeta values in \cite{BJOP02}.
The second generalization, called \emph{$N$-congruent multiple zeta values} ($N$-CMZVs in short) are real numbers defined by the following series
\begin{equation}
    \label{NMZV}
    \zeta_{(\alpha_1, \dots, \alpha_r)}^{\bmod N}{(k_1, \dots, k_r)}
    = \sum_{\substack{n_1 > \dots > n_r > 0 \\ n_j \in \alpha_j \ 1 \leq j \leq r}} \frac{1}{n_1^{k_1} \cdots n_r^{k_r}},
\end{equation}
where $r, k_1, \dots, k_r \in \Z_{>0}$ with $k_1 \neq 1$ and $\alpha_1, \dots, \alpha_r \in \Z/N\Z$. They are also called the \emph{level $N$ multiple zeta value} in \cite{YZ16}. Denote by $\iota : \llbracket 1, N\rrbracket \to \Z / N \Z$ the substitution that identifies each element of $\llbracket 1, N\rrbracket$ with its equivalence class.
It has been established in \cite[(8)]{YZ16} that $N$-CMZVs \eqref{NMZV} are related to $N$-MPVs \eqref{NMPV} by the following formula 
\begin{equation}
    \label{NMPV_NCMZV} 
    \zeta_{(\alpha_1, \dots, \alpha_r)}^{\bmod N}{(k_1, \dots, k_r)}
    = \frac{1}{N^r} \sum_{m_1 = 1}^N \cdots \sum_{m_r = 1}^N \zeta_N^{-(m_1 \iota^{-1}(\alpha_1) + \cdots + m_r \iota^{-1}(\alpha_r))} \Li_{(k_1, \dots, k_r)}(\zeta_N^{m_1}, \dots, \zeta_N^{m_r}), 
\end{equation}
where $\zeta_N := \exp\left(\frac{2\pi\mathrm{i}}{N}\right)$. \\

The values \eqref{NMPV} can be expressed as an iterated integral as follows (see for example \cite{Gon98})
\begin{equation}
    \label{MPV_int}
    \Li_{(k_1, \dots, k_r)}(\zeta_1, \dots, \zeta_r) = \int_0^1 \Omega_0^{k_1-1} \Omega_{\zeta_1} \Omega_0^{k_2-1} \Omega_{\zeta_1 \zeta_2} \cdots \Omega_0^{k_r-1} \Omega_{\zeta_1 \cdots \zeta_r}, 
\end{equation}
where $\Omega_0 = \frac{dt}{t}$ and $\Omega_\zeta = \frac{dt}{\zeta^{-1}-t}$, for any root of unity $\zeta$. Thanks to formula \eqref{NMPV_NCMZV} and the iterated integral expression \eqref{MPV_int}, one deduces the following identity (see \cite[p.185]{YZ16})
\[
    \zeta_{(\alpha_1, \dots, \alpha_r)}^{\bmod N}{(k_1, \dots, k_r)} = \frac{1}{N^r} \int_0^1 \omega^{k_1-1} \omega_{\alpha_1 - \alpha_2} \cdots \omega^{k_{r-1}-1} \omega_{\alpha_{r-1} - \alpha_r} \omega^{k_r-1} \omega_{\alpha_r}, 
\]
where $\omega = \frac{dt}{t}$ and $\omega_\alpha = \frac{N t^{\iota^{-1}(\alpha) - 1}dt}{1-t^N}$ for any $\alpha \in \Z / N \Z$. One may also check that \cite[p.185]{YZ16}
\begin{equation}
    \label{eq:diff_1forms}
    \omega = \Omega_0 \text{ and } \omega_\alpha = \sum_{m=1}^N \zeta_N^{-m \iota^{-1}(\alpha)} \Omega_{\zeta_N^m}.
\end{equation}

Using iterated sums and iterated integral expression, one deduces linear and algebraic relations between $N$-MPVs called the (regularized) \emph{double shuffle} relations. To describe these relations Hoffmann \cite{Hof97} introduced a word algebra setting whose letters are free noncommutative variables assimilated to the differential $1$-forms $\Omega_0$ and $\Omega_\zeta$ ($\zeta \in \mu_N$). This setting has been utilized to provide two kinds of algebraic frameworks for these relations. The first one, introduced by Racinet \cite{Rac02} uses noncommutative formal power series Hopf algebras. This framework introduced a $\Q$-Lie algebra $\dmr_0^{\mu_N}$, which is a key ingredient in Racinet's main result on the torsor structure of (regularized) double shuffle and distribution relations. The second framework, introduced by Arakawa-Kaneko \cite{AK04}, was inspired by Ihara-Kaneko-Zagier \cite{IKZ} ($N=1$) and utilizes noncommutative free polynomial algebras. The equivalence between the two frameworks is related to dual nature of the relation between algebras and Hopf algebras in general.
The relation \eqref{NMPV_NCMZV} suggests that $N$-CMZVs also satisfy (regularized) double shuffle relations. In fact, this has been introduced by Yuan-Zhao \cite{YZ16} and written down precisely by Kanno in \cite{K24} thanks to the word algebra setting whose letters are the free noncommutative variables associated to the differential $1$-forms $\omega$ and $\omega_\alpha$ ($\alpha \in \Z/N\Z$). \\

In this paper, we utilize the identities \eqref{eq:diff_1forms} to construct an algebra isomorphism between the free noncommutative series algebras with coefficients in the cyclotomic field $\Q(\mu_N)$ associated to each set of differential $1$-forms (see Lemma \ref{lem:iso_mathcalF}). This enables us to introduce a dual equivalent of the formalism of \cite{K24}, which facilitates the formulation of double shuffle relations among $N$-CMZVs within the formalism of \cite{Rac02}. We then show that this algebra isomorphism is in fact a Hopf algebra isomorphism with respect to the newly introduced coproducts and the coproducts of \cite{Rac02}, thus establishing the relations between both formalisms. After that, we define a $\Q$-Lie algebra $\dmr_0^{[N]}$ related to the (regularized) double shuffle relations among $N$-CMZVs and establish the following result:

\begin{thm}[Theorem \ref{main theorem}]
    There is a $\Q(\mu_N)$-Lie algebra isomorphism
    \begin{equation*}
        \Q(\mu_N) \ \hat{\otimes}_\Q \ \dmr_0^{[N]} \xrightarrow{\simeq} \Q(\mu_N) \ \hat{\otimes}_\Q \ \dmr_0^{\mu_N}.
    \end{equation*}
\end{thm}

Finally, we focus our interest in the $\Q$-forms of these Lie algebra. In order to do so, we introduce a suitable action of the Galois group $\mathrm{Gal}(\Q(\mu_N)/\Q)$, which enables us to establish the following result:

\begin{thm}[Theorem \ref{Galois Descent theorem}]
    There is a $\Q$-Lie algebra isomorphism between $\dmr_0^{\mu_N}$ (resp. $\dmr_0^{[N]}$) and the invariant subspace of $\Q(\mu_N) \ \hat{\otimes}_\Q \ \dmr_0^{[N]}$ (resp. $\Q(\mu_N) \ \hat{\otimes}_\Q \ \dmr_0^{\mu_N}$) under the action of $\mathrm{Gal}(\Q(\mu_N)/\Q)$.
\end{thm}

%%%%%%%%%%%%%%%%%%%%%%%%%%%%%%%%%%%%%%%%%%%%%%%%%%%%%%%%%%%

\paragraph{\textbf{Acknowledgments}}
This project was partially supported by first author's JSPS KAKENHI Grants 24K00520 and 24K21510 and second author's JSPS KAKENHI Grant 23KF0230. The authors would like to thank Benjamin Enriquez for fruitful comments on this project. \\

%%%%%%%%%%%%%%%%%%%%%%%%%%%%%%%%%%%%%%%%%%%%%%%%%%%%%%%%%%%

\paragraph{\textbf{Notation}}
Throughout this paper, let $N \geq 3$ be a positive integer and $\mu_N$ be the group of complex roots of unity. Denote by $\iota : \llbracket 1, N\rrbracket \to \Z / N \Z$ the substitution that identifies each element of $\llbracket 1, N\rrbracket$ with its equivalence class. Finally, let $\Q$ be a commutative $\Q$-algebra with unit.

\section{Algebraic framework related to \texorpdfstring{$N$}{N}-MPVs}
This section reviews the algebras in which one expresses shuffle and harmonic relations between $N$-MPVs. Their dual counterparts are Hopf algebras in which Racinet formulates \cite{Rac02} the cyclotomic double shuffle Lie algebra $\dmr_0^{\mu_N}$.

\subsection{Shuffle and harmonic algebras}
Denote by $\Q\langle X\rangle$ the free noncommutative polynomial \linebreak $\Q$-algebra with unit over the alphabet $X := \{x_0, x_\zeta \mid \zeta \in \mu_N \}$. \newline
Consider the set $Y := \{y_{k, \zeta} \mid (k, \zeta) \in \Z_{>0} \times \mu_N \}$ and denote by $\Q\langle Y\rangle$ the free noncommutative polynomial $\Q$-algebra with unit over $Y$.
The $\Q$-algebra $\Q\langle Y\rangle$ can be seen as a subalgebra of $\Q\langle X\rangle$ thanks to the injective algebra morphism $\Q\langle Y\rangle \hookrightarrow \Q\langle X\rangle$ given by $y_{k, \zeta} \mapsto x_0^{k-1} x_\zeta$, for any $(k, \zeta) \in \Z_{>0} \times \mu_N$.

\begin{lem}[\cite{AK04, Zha10}]
    \begin{enumerate}[label=(\alph*), leftmargin=*]
        \item The $\Q$-linear space $\Q\langle X\rangle$ is equipped with a commutative $\Q$-algebra structure with respect to the \emph{shuffle product} given inductively by
        \[
            1 \ \sh \ w = w \ \sh \ 1 = w,
        \]
        for a word $w$ in $\Q\langle X\rangle$, and
        \[
            u w_{1} \ \sh \ v w_{2} = u(w_{1} \ \sh \ v w_{2}) + v(u w_{1} \ \sh \ w_{2}),
        \]
        for words $w_1, w_2$ in $\Q\langle X\rangle$ and $u, v \in X$; and extending by $\mathbb{Q}$-bilinearity.
        \item The $\Q$-linear space $\Q\langle Y\rangle$ is equipped with a commutative $\Q$-algebra structure with respect to the \emph{harmonic product} (a.k.a. stuffle product) given inductively by
        \[
            1 \ast w = w \ast 1 = w,
        \]
        for a word $w$ in $\Q\langle Y\rangle$, and
        \[
            y_{k_1, \zeta_1} w_{1} \ast y_{k_2, \zeta_2} w_{2} = y_{k_1, \zeta_1}(w_{1} \ast y_{k_2, \zeta_2} w_{2}) + y_{k_2, \zeta_2} (y_{k_1, \zeta_1} w_{1} \ast w_{2}) + y_{k_1+k_2, \zeta_1 \zeta_2} (w_{1} \ast w_{2})
        \]
        for words $w_1, w_2$ in $\Q\langle Y\rangle$, $(k_1, \zeta_1), (k_2, \zeta_2) \in \Z_{>0} \times \mu_N$; and extending by $\Q$-bilinearity. 
    \end{enumerate}
\end{lem}

These algebraic structures provide a framework for formulating the shuffle and harmonic product formulas 
for $N$-MPVs, as elaborated in \cite{Hof97}.

\subsection{Shuffle and harmonic Hopf algebras}
Denote by $\Q\langle\langle X\rangle\rangle$ the free noncommutative series $\Q$-algebra with unit over $X$ and let $\Q\langle\langle Y\rangle\rangle$ be the free noncommutative series $\Q$-algebra with unit over $Y$. \newline
The $\Q$-algebra $\Q\langle\langle Y\rangle\rangle$ can be seen as a subalgebra of $\Q\langle\langle X\rangle\rangle$ thanks to the injective algebra morphism $\Q\langle\langle Y\rangle\rangle \hookrightarrow \Q\langle\langle X\rangle\rangle$ given by $y_{k, \zeta} \mapsto x_0^{k-1} x_\zeta$, for any $(k, \zeta) \in \Z_{>0} \times \mu_N$.
Recall the direct sum decomposition (of $\Q$-linear subspaces) (see (\cite[§2.2.5]{Rac02})
\[
    \Q\langle\langle X\rangle\rangle = \Q\langle\langle Y\rangle\rangle \oplus \Q\langle\langle X\rangle\rangle x_0.
\]
Let then $\pi_Y : \Q\langle\langle X\rangle\rangle = \Q\langle\langle Y\rangle\rangle \oplus \Q\langle\langle X\rangle\rangle x_0 \twoheadrightarrow \Q\langle\langle Y\rangle\rangle$ be the projection from $\Q\langle\langle X\rangle\rangle$ to $\Q\langle\langle Y\rangle\rangle$, that is, the surjective $\Q$-module morphism such that it is the identity on $\Q\langle\langle Y\rangle\rangle$ and maps any element of $\Q\langle\langle X\rangle\rangle x_0$ to $0$.\newline
Let $\mathbf{p}$ be the $\Q$-linear automorphism of $\Q\langle\langle X\rangle\rangle$ given by (\cite[§2.2.7]{Rac02})
\[
    \mathbf{p}(x_0^{k_1-1}x_{\zeta_1}x_0^{k_2-1}x_{\zeta_2} \cdots x_0^{k_r-1}x_{\zeta_r} x_0^{k_{r+1}-1}) = x_0^{k_1-1}x_{\zeta_1}x_0^{k_2-1}x_{\zeta_1\zeta_2} \cdots x_0^{k_r-1}x_{\zeta_1 \cdots \zeta_r}x_0^{k_{r+1}-1},
\]
for $r \in \Z_{\geq 0}$, $k_1, \dots, k_{r+1} \in \Z_{>0}$ and $\zeta_1, \dots, \zeta_r \in \mu_N$. \newline
Denote by $\mathcal{L}$ either $X$ or $Y$. Define the pairing
\[
    \Q\langle\langle\mathcal{L}\rangle\rangle \otimes \Q\langle\mathcal{L}\rangle \to \Q, \, \psi \otimes w \mapsto (\psi \mid w),
\]
where $(\psi \mid w)$ denotes the coefficient of the word $w$ in $\psi$ (and extend by linearity). Using this pairing, one obtains the following result:
\begin{lem}
    \begin{enumerate}[label=(\alph*), leftmargin=*]
        \item The $\Q$-algebra $\Q\langle\langle X\rangle\rangle$ is equipped with a Hopf algebra structure with respect to the \emph{shuffle coproduct}, which is the algebra morphism $\widehat{\Delta}_\sh : \Q\langle\langle X\rangle\rangle \to \Q\langle\langle X\rangle\rangle^{\hat{\otimes} 2}$ given by
        \[
            x_0 \mapsto x_0 \otimes 1 + 1 \otimes x_0 \text{ and } x_\zeta \mapsto x_\zeta \otimes 1 + 1 \otimes x_\zeta, \text{ for } \zeta \in \mu_N.  
        \]
        \item The $\Q$-algebra $\Q\langle\langle Y\rangle\rangle$ is equipped with a Hopf algebra structure with respect to the \emph{harmonic coproduct}, which is the algebra morphism $\widehat{\Delta}_\ast : \Q\langle\langle Y\rangle\rangle \to \Q\langle\langle Y\rangle\rangle^{\hat{\otimes} 2}$ given by
        \[
            y_{k, \zeta} \mapsto y_{k, \zeta} \otimes 1 + 1 \otimes y_{k, \zeta} + \sum_{\substack{k_1 + k_2 = k \\ \zeta_1 \zeta_2 = \zeta}} y_{k_1,\zeta_1} \otimes y_{k_2,\zeta_2}, \text{ for } (k, \zeta) \in \Z_{>0} \times \mu_N.  
        \]
    \end{enumerate}
\end{lem}
\begin{proof}
    Setting $\bullet$ to be either $\sh$ or $\ast$, this follows from the identity (for explicit proof the reader may refer to \cite[Lemma A.3]{BY24})
    \begin{align*}
        \widehat{\Delta}_{\bullet}(\psi) = \sum_{u,v \text{ words in } \mathcal{L}} (\psi \mid u \bullet v) u \otimes v,
    \end{align*}
    for any $\psi \in \Q\langle\langle\mathcal{L}\rangle\rangle$.
\end{proof}

\subsection{Racinet's cyclotomic double shuffle Lie algebra \texorpdfstring{$\dmr_0^{\mu_N}$}{dmr0muN}}
For $\psi \in \Q\langle\langle X\rangle\rangle$, let $d_{\psi}$ be the derivation of $\Q\langle\langle X\rangle\rangle$ given by \cite[§3.1.12.2]{Rac02}
\begin{align}\label{eq:derivation d}
        x_0 \mapsto 0, \text{ and } \, x_\zeta \mapsto [x_\zeta, t_\zeta(\psi)] \text{ for } \zeta \in \mu_N,
\end{align}
where $t_\zeta$ is the algebra automorphism of $\Q\langle\langle X\rangle\rangle$ given by
\begin{align}\label{eq: auto t}
        x_0 \mapsto x_0, \text{ and } \, x_{\eta} \mapsto x_{\zeta\eta}, \text{ for } \eta \in \mu_N.
\end{align}
We then define a Lie bracket on $\Q\langle\langle X\rangle\rangle$ as follows \cite[§3.1.10.2]{Rac02}:
\begin{equation}\label{eq: bracket for dmr}
    \langle \psi_1, \psi_2 \rangle := d_{\psi_1}(\psi_2) - d_{\psi_2}(\psi_1) + [\psi_1, \psi_2],
\end{equation}
for any $\psi_1, \psi_2 \in \Q\langle\langle X\rangle\rangle$.
\begin{defn}[{\cite[Definitions 3.3.1 and 3.3.8]{Rac02}}]
    \label{dmrmuN}
    We define $\dmr_0^{\mu_N}$ to be the set of elements $\psi \in \Q\langle\langle X\rangle\rangle$ such that
    \begin{enumerate}[label=(\roman*), leftmargin=*, itemsep=2mm]
        \begin{multicols}{2}
        \item \label{dmrmuN_i} $(\psi \mid x_0) = (\psi \mid x_1) = 0$;
        \item \label{dmrmuN_ii} $\widehat{\Delta}_\sh(\psi) = \psi \otimes 1 + 1 \otimes \psi$;
        \item\label{dmrmuN_iii}\footnote{In \cite{Rac02}, this condition is given as $\left(\psi_\ast \mid x_0^{k-1}x_\zeta\right) = (-1)^{k-1} \left(\psi_\ast \mid x_0^{k-1}x_{\zeta^{-1}}\right)$ for $(k, \zeta) \in \Z_{>0} \times \mu_N$, but thanks to \cite[Propositions 3.3.3 and 3.3.7]{Rac02}, it is enough to have \ref{dmrmuN_iii} for $k=1$ and any $\zeta \in \mu_N$ since this identity is always true for all the other cases. By definition of $\psi_\ast$ this is equivalent to the stated identity.} $\left(\psi \mid x_\zeta - x_{\zeta^{-1}}\right) = 0$ for any $\zeta \in \mu_N$;
        \item \label{dmrmuN_iv} $\widehat{\Delta}_\ast(\psi_\ast) = \psi_\ast \otimes 1 + 1 \otimes \psi_\ast$;
        \end{multicols}
    \end{enumerate}
    where
    \[
        \psi_\ast = \pi_Y \circ \mathbf{p}^{-1}(\psi) + \sum_{n \geq 2} \frac{(-1)^{n-1}}{n} (\psi \mid x_0^{n-1} x_1) y_{1,1}^n \in \Q\langle\langle Y\rangle\rangle.
    \]
\end{defn}

\begin{prop}[{\cite[Proposition 4.A.i)]{Rac02}}]
    The pair $(\dmr_0^{\mu_N}, \langle \cdot, \cdot \rangle)$ is a Lie subalgebra of $(\Q\langle\langle X\rangle\rangle, \langle \cdot, \cdot \rangle)$. 
\end{prop}

%%%%%%%%%%%%%%%%%%%%%%%%%%%%%%%%%%%%%%%%%%%%%%%%%%%%%%%%%%%%%%%%%%%%%%
\section{Algebraic framework related to \texorpdfstring{$N$}{N}-CMZVs}
Paralleling the approach in the previous section, we first recall from \cite{K24} the algebras in which one expresses shuffle and harmonic relations between $N$-CMZVs. Then we introduce here their dual counterparts which are Hopf algebras in which we formulate a congruent double shuffle Lie algebra $\dmr_0^{[N]}$.

\subsection{Shuffle and harmonic algebras}
Denote by $\Q\langle\widetilde{X}\rangle$ the free noncommutative polynomial \linebreak $\Q$-algebra with unit over the alphabet $\widetilde{X} := \{\tilde{x}, \tilde{x}_\alpha \mid \alpha \in \Z / N \Z \}$. \newline
Consider the set $\widetilde{Y} := \{\tilde{y}_{k, \alpha} \mid (k, \alpha) \in \Z_{>0} \times \Z / N \Z \}$ and denote by $\Q\langle\widetilde{Y}\rangle$ the free noncommutative polynomial $\Q$-algebra with unit over $\widetilde{Y}$.
The $\Q$-algebra $\Q\langle\widetilde{Y}\rangle$ can be seen as a subalgebra of $\Q\langle\widetilde{X}\rangle$ thanks to the injective algebra morphism $\Q\langle\widetilde{Y}\rangle \hookrightarrow \Q\langle\widetilde{X}\rangle$ given by $\tilde{y}_{k, \alpha} \mapsto \tilde{x}^{k-1} \tilde{x}_\alpha$, for $(k, \alpha) \in \Z_{>0} \times \Z / N \Z$.

\begin{lem}[{\cite[p. 5 and 6]{K24}}]
    \begin{enumerate}[label=(\alph*), leftmargin=*]
        \item The $\Q$-linear space $\Q\langle\widetilde{X}\rangle$ is equipped with a commutative $\Q$-algebra structure with respect to the \emph{shuffle product} given inductively by
        \[
            1 \ \tilde{\sh} \ \tilde{w} = \tilde{w} \ \tilde{\sh} \ 1 = \tilde{w},
        \]
        for a word $\tilde{w}$ in $\Q\langle\widetilde{X}\rangle$; and
        \[
            \tilde{u} \tilde{w}_{1} \ \tilde{\sh} \ \tilde{v} \tilde{w}_{2} = \tilde{u}(\tilde{w}_{1} \ \tilde{\sh} \ \tilde{v} \tilde{w}_{2}) + \tilde{v}(\tilde{u} \tilde{w}_{1} \ \tilde{\sh} \ \tilde{w}_{2}),
        \]
        for words $\tilde{w}_1, \tilde{w}_2$ in $\Q\langle\widetilde{X}\rangle$ and $\tilde{u}, \tilde{v} \in \widetilde{X}$, and extending by $\mathbb{Q}$-bilinearity.
        \item The $\Q$-linear space $\Q\langle\widetilde{Y}\rangle$ is equipped with a commutative $\Q$-algebra structure with respect to the \emph{harmonic product} given inductively by
        \[
            1 \ \tilde{\ast} \ \tilde{w} = \tilde{w} \ \tilde{\ast} \ 1 = \tilde{w},
        \]
        for a word $\tilde{w}$ in $\Q\langle\widetilde{Y}\rangle$, and
        \begin{align*}
            &\tilde{y}_{k_1, \alpha_1} w_{1} \ \tilde{\ast} \ \tilde{y}_{k_2, \alpha_2} w_{2} \\
            & =
            \begin{cases}
            \tilde{y}_{k_1, \alpha_1}(\tilde{w}_{1} \tilde{\ast} \tilde{y}_{k_2, \alpha_2} \tilde{w}_{2}) + \tilde{y}_{k_2, \alpha_2} (\tilde{y}_{k_1, \alpha_1} \tilde{w}_{1} \tilde{\ast} \tilde{w}_{2}) & \text{ if } \alpha_1 \neq \alpha_2 \\
            \tilde{y}_{k_1, \alpha}(\tilde{w}_{1} \tilde{\ast} \tilde{y}_{k_2, \alpha} \tilde{w}_{2}) + \tilde{y}_{k_2, \alpha} (\tilde{y}_{k_1, \alpha} \tilde{w}_{1} \tilde{\ast} \tilde{w}_{2}) + \tilde{y}_{k_1+k_2, \alpha} (\tilde{w}_{1} \tilde{\ast} \tilde{w}_{2}) & \text{ if } \alpha=\alpha_1=\alpha_2
            \end{cases}
        \end{align*}
        for words $\tilde{w}_1, \tilde{w}_2$ in $\Q\langle\widetilde{Y}\rangle$, $(k_1 ,\alpha_1), (k_2 ,\alpha_2) \in \Z_{>0} \times \Z / N \Z$; and extending by $\Q$-bilinearity. 
    \end{enumerate}
\end{lem}

These algebraic structures provide a framework for formulating the shuffle and  harmonic product formulas 
for $N$-CMZVs, as elaborated in \cite{K24}.

\subsection{Shuffle and harmonic Hopf algebras}
Denote by $\Q\langle\langle\widetilde{X}\rangle\rangle$ the free noncommutative series $\Q$-algebra with unit over $\widetilde{X}$.
Denote by $\Q\langle\langle\widetilde{Y}\rangle\rangle$ the free noncommutative series $\Q$-algebra with unit over $\widetilde{Y}$. \newline
The $\Q$-algebra $\Q\langle\langle\widetilde{Y}\rangle\rangle$ can be seen as a subalgebra of $\Q\langle\langle\widetilde{X}\rangle\rangle$ thanks to the injective algebra morphism $\Q\langle\langle\widetilde{Y}\rangle\rangle \hookrightarrow \Q\langle\langle\widetilde{X}\rangle\rangle$ given by $\tilde{y}_{k, \alpha} \mapsto \tilde{x}^{k-1} \tilde{x}_\alpha$, for $(k, \alpha) \in \Z_{>0} \times \Z / N \Z$.
One checks that we have the direct sum decomposition (of $\Q$-submodules)
\[
    \Q\langle\langle\widetilde{X}\rangle\rangle = \Q\langle\langle \widetilde{Y}\rangle\rangle \oplus \Q\langle\langle\widetilde{X}\rangle\rangle \tilde{x}.
\]
Let then $\pi_{\widetilde{Y}} : \Q\langle\langle\widetilde{X}\rangle\rangle = \Q\langle\langle\widetilde{Y}\rangle\rangle \oplus \Q\langle\langle\widetilde{X}\rangle\rangle \tilde{x} \twoheadrightarrow \Q\langle\langle\widetilde{Y}\rangle\rangle$ be the projection from $\Q\langle\langle\widetilde{X}\rangle\rangle$ to $\Q\langle\langle\widetilde{Y}\rangle\rangle$, that is, the surjective $\Q$-module morphism such that it is the identity on $\Q\langle\langle\widetilde{Y}\rangle\rangle$ and maps any element of $\Q\langle\langle\widetilde{X}\rangle\rangle \tilde{x}$ to $0$.\newline
Let $\widetilde{\mathbf{q}}$ be the $\Q$-linear automorphism of $\Q\langle\langle\widetilde{X}\rangle\rangle$ given by
\[
    \mbox{\small $\widetilde{\mathbf{q}}(\tilde{x}^{k_1-1}\tilde{x}_{\alpha_1} \cdots \tilde{x}^{k_{r-1}-1}\tilde{x}_{\alpha_{r-1}} \tilde{x}^{k_r-1}\tilde{x}_{\alpha_r} \tilde{x}^{k_{r+1}-1}) = \tilde{x}^{k_1-1}\tilde{x}_{\alpha_1 - \alpha_2} \cdots \tilde{x}^{k_{r-1}-1}\tilde{x}_{\alpha_{r-1} - \alpha_r} \tilde{x}^{k_r-1}\tilde{x}_{\alpha_r} \tilde{x}^{k_{r+1}-1}$},
\]
for $r \in \Z_{\geq 0}$, $k_1, \dots, k_{r+1} \in \Z_{>0}$ and $\alpha_1, \dots, \alpha_r \in \Z / N \Z$.\newline
Denote by $\widetilde{\mathcal{L}}$ either $\widetilde{X}$ or $\widetilde{Y}$. Define the pairing
\[
    \Q\langle\langle\widetilde{\mathcal{L}}\rangle\rangle \otimes \Q\langle\widetilde{\mathcal{L}}\rangle \to \Q, \, \widetilde{\psi} \otimes \tilde{w} \mapsto (\widetilde{\psi} \mid \tilde{w}),
\]
where $(\widetilde{\psi} \mid \tilde{w})$ denotes the coefficient of the word $\tilde{w}$ in $\widetilde{\psi}$ (and extend by linearity). Using this pairing, one obtains the following result:
\begin{lem}
    \begin{enumerate}[label=(\alph*), leftmargin=*]
        \item The $\Q$-algebra $\Q\langle\langle\widetilde{X}\rangle\rangle$ is equipped with a Hopf algebra structure with respect to the \emph{shuffle coproduct}, which is the algebra morphism $\widehat{\Delta}_{\tilde{\sh}} : \Q\langle\langle\widetilde{X}\rangle\rangle \to \Q\langle\langle\widetilde{X}\rangle\rangle^{\hat{\otimes} 2}$ given by
        \[
            \tilde{x} \mapsto \tilde{x} \otimes 1 + 1 \otimes \tilde{x} \text{ and } \tilde{x}_\alpha \mapsto \tilde{x}_\alpha \otimes 1 + 1 \otimes \tilde{x}_\alpha, \text{ for } \alpha \in \Z / N \Z.  
        \]
        \item The $\Q$-algebra $\Q\langle\langle\widetilde{Y}\rangle\rangle$ is equipped with a Hopf algebra structure with respect to the \emph{harmonic coproduct}, which is the algebra morphism $\widehat{\Delta}_{\tilde{\ast}} : \Q\langle\langle\widetilde{Y}\rangle\rangle \to \Q\langle\langle\widetilde{Y}\rangle\rangle^{\hat{\otimes} 2}$ given by
        \[
            \tilde{y}_{k,\alpha} \mapsto \tilde{y}_{k,\alpha} \otimes 1 + 1 \otimes \tilde{y}_{k,\alpha} + \sum_{k_1 + k_2 = k} \tilde{y}_{k_1,\alpha} \otimes \tilde{y}_{k_2,\alpha}, \text{ for } (k, \alpha) \in \Z_{>0} \times \Z / N \Z.  
        \]
    \end{enumerate}
\end{lem}
\begin{proof}
    Setting $\tilde{\bullet}$ to be either $\tilde{\sh}$ or $\tilde{\ast}$, this follows from the identity (for explicit proof the reader may refer to \cite[Lemma A3]{BY24})
    \begin{align*}
        \widehat{\Delta}_{\tilde{\bullet}}(\widetilde{\psi}) = \sum_{\tilde{u},\tilde{v} \text{ words in } \widetilde{\mathcal{L}}} (\widetilde{\psi} \mid \tilde{u} \ \tilde{\bullet} \ \tilde{v}) \tilde{u} \otimes \tilde{v},
    \end{align*}
    for any $\widetilde{\psi} \in \Q\langle\langle\widetilde{\mathcal{L}}\rangle\rangle$.
\end{proof}

\section{Congruent double shuffle Lie algebra \texorpdfstring{$\dmr_0^{[N]}$}{dmr0N} }%\label{sec:main results}
Motivated by identity \eqref{NMPV_NCMZV}, we relate the formalisms of double shuffle relations for $N$-MPVs and $N$-CMZVs. 
Subsequently, we introduce a congruent double shuffle Lie algebra
$\dmr_0^{[N]}$ and demonstrate that it provides a new $\Q$-form of $\dmr^{\mu_N}_0$. 
Our main result is a $G_N$-invariance characterization theorem, which establishes how these two $\mathbb{Q}$-structures can be reconstructed from each other under Galois action, thus elucidating their intricate relationship.

\subsection{Comparison between the frameworks related to \texorpdfstring{$N$}{N}-MPVs and \texorpdfstring{$N$}{N}-CMZVs}
\begin{lem}
    \label{lem:iso_mathcalF}
    The $\Q(\mu_N)$-algebra morphism ${\mathcal F} : \Q(\mu_N)\langle\langle\widetilde{X}\rangle\rangle \to \Q(\mu_N)\langle\langle X\rangle\rangle$ given by
    \[
        \tilde{x} \mapsto x_0 \text{ and } \tilde{x}_\alpha \mapsto \sum_{m=1}^N \zeta_N^{-m \iota^{-1}(\alpha)} x_{\zeta_N^m} \text{ for } \alpha \in \Z / N \Z,
    \]
    is an isomorphism whose inverse ${\mathcal F}^{-1}$is given by
    \[
        x_0 \mapsto \tilde{x} \text{ and } x_{\zeta_N^m} \mapsto \frac{1}{N} \sum_{a=1}^N \zeta_N^{a m} \tilde{x}_{\iota(a)} \text{ for } m \in \llbracket 1, N\rrbracket.
    \]
\end{lem}
\begin{proof}
    This follows from a direct computation using the identity
    \begin{equation*}
        %\label{sum_roots}
        \sum_{j=1}^N \zeta_{N}^{Mj} =
        \begin{cases}
            N & \text{if }  N \ \mid \ M\\
            0 & \text{otherwise}
        \end{cases}
    \end{equation*}
\end{proof}

\begin{lem}
    The $\Q(\mu_N)$-algebra isomorphism ${\mathcal F} : \Q(\mu_N)\langle\langle\widetilde{X}\rangle\rangle \to \Q(\mu_N)\langle\langle X\rangle\rangle$ restricts to the $\Q(\mu_N)$-algebra isomorphism ${\mathcal F}_Y : \Q(\mu_N)\langle\langle\widetilde{Y}\rangle\rangle \to \Q(\mu_N)\langle\langle Y\rangle\rangle$ given by
    \[
        \tilde{y}_{k, \alpha} \mapsto \sum_{m=1}^N \zeta_N^{-m \iota^{-1}(\alpha)} y_{k, \zeta_N^m} \text{ for } (k, \alpha) \in \Z_{>0} \times \Z / N \Z.
    \]
\end{lem}
\begin{proof}
    This follows from the fact ${\mathcal F}_Y : \Q(\mu_N)\langle\langle\widetilde{Y}\rangle\rangle \to \Q(\mu_N)\langle\langle Y\rangle\rangle$ is the unique $\Q(\mu_N)$-algebra morphism such that the following diagram
    \[\begin{tikzcd}
         \Q(\mu_N)\langle\langle \widetilde{Y}\rangle\rangle \ar[rr, "{\mathcal F}_Y"] \ar[d, hook] && \Q(\mu_N)\langle\langle Y\rangle\rangle \ar[d, hook'] \\
        \Q(\mu_N)\langle\langle \widetilde{X}\rangle\rangle \ar[rr, "{\mathcal F}"] && \Q(\mu_N)\langle\langle X\rangle\rangle
    \end{tikzcd}\]
    commutes, whose inverse is the unique $\Q(\mu_N)$-algebra morphism $\Q(\mu_N)\langle\langle Y\rangle\rangle \to \Q(\mu_N)\langle\langle\widetilde{Y}\rangle\rangle$ such that the following diagram
    \[\begin{tikzcd}
         \Q(\mu_N)\langle\langle Y\rangle\rangle \ar[rr] \ar[d, hook] && \Q(\mu_N)\langle\langle \widetilde{Y}\rangle\rangle \ar[d, hook'] \\
        \Q(\mu_N)\langle\langle X\rangle\rangle \ar[rr, "{\mathcal F}^{-1}"] && \Q(\mu_N)\langle\langle \widetilde{X}\rangle\rangle
    \end{tikzcd}\]
    commutes.
\end{proof}

\begin{lem}
    \label{lem:diagpi_diagqp}
    The following diagrams of $\Q(\mu_N)$-linear maps are commutative
    \begin{equation}\label{diag:piY_pitildeY}\begin{tikzcd}
        \Q(\mu_N)\langle\langle\widetilde{X}\rangle\rangle \ar[rr, "{\mathcal F}"] \ar[d, "\pi_{\widetilde{Y}}"'] &&  \ar[d, "\pi_Y"] \Q(\mu_N)\langle\langle X\rangle\rangle \\
        \Q(\mu_N)\langle\langle \widetilde{Y}\rangle\rangle \ar[rr, "{\mathcal F}_Y"] && \Q(\mu_N)\langle\langle Y\rangle\rangle
    \end{tikzcd}\end{equation}
    and
    \begin{equation}\label{diag:p_tildeq}\begin{tikzcd}
        \Q(\mu_N)\langle\langle\widetilde{X}\rangle\rangle \ar[rr, "{\mathcal F}"] \ar[d, "\widetilde{\mathbf{q}}"'] && \ar[d, "\mathbf{p}"] \Q(\mu_N)\langle\langle X\rangle\rangle \\
        \Q(\mu_N)\langle\langle \widetilde{X}\rangle\rangle \ar[rr, "{\mathcal F}"] && \Q(\mu_N)\langle\langle X\rangle\rangle
    \end{tikzcd}\end{equation}
\end{lem}
\begin{proof}
    From the definition, an immediate computation enables one to prove the commutativity of Diagram \eqref{diag:piY_pitildeY}. Let us prove the commutativity of Diagram \eqref{diag:p_tildeq}. It is enough to prove it on a basis of $\Q(\mu_N)\langle\langle\widetilde{X}\rangle\rangle$. Let $\tilde{x}^{k_1-1}\tilde{x}_{\alpha_1} \tilde{x}^{k_2-1}\tilde{x}_{\alpha_2} \cdots \tilde{x}^{k_r-1}\tilde{x}_{\alpha_r} \tilde{x}^{k_{r+1}-1}$ with $r \in \Z_{\geq 0}$, $k_1, \dots, k_{r+1} \in \Z_{>0}$ and $\alpha_1, \dots, \alpha_r \in \Z / N \Z$. We have
    \begin{align*}
    & \begin{aligned} {\mathcal F} \circ \widetilde{\mathbf{q}}(\tilde{x}^{k_1-1}\tilde{x}_{\alpha_1} \cdots \tilde{x}^{k_{r-1}-1}\tilde{x}_{\alpha_{r-1}} \tilde{x}^{k_r-1}&\tilde{x}_{\alpha_r} \tilde{x}^{k_{r+1}-1}) \\
    & = {\mathcal F}(\tilde{x}^{k_1-1}\tilde{x}_{\alpha_1 - \alpha_2} \cdots \tilde{x}^{k_{r-1}-1}\tilde{x}_{\alpha_{r-1} - \alpha_r} \tilde{x}^{k_r-1}\tilde{x}_{\alpha_r} \tilde{x}^{k_{r+1}-1}) \end{aligned} \\
    & \begin{aligned}= x_0^{k_1-1} \sum_{m_1^\prime=1}^N \zeta_N^{-m_1^\prime \iota^{-1}(\alpha_1-\alpha_2)} x_{\zeta_N^{m_1^\prime}} \cdots x_0^{k_{r-1}-1} \sum_{m_{r-1}^\prime=1}^N \zeta_N^{-m}&{}^{_{r-1}^\prime\iota^{-1}(\alpha_{r-1}-\alpha_r)} x_{\zeta_N^{m_{r-1}^\prime}} \\
    & x_0^{k_r-1} \sum_{m_r^\prime=1}^N \zeta_N^{-m_r^\prime \iota^{-1}(\alpha_r)} x_{\zeta_N^{m_r^\prime}} x_0^{k_{r+1}-1} \end{aligned} \\
    & \begin{aligned}= \sum_{m_1^\prime=1}^N \sum_{m_2^\prime=1}^N \cdots \sum_{m_r^\prime=1}^N \zeta_N^{-m_1^\prime \iota^{-1}(\alpha_1) + (m_1^\prime-m_2^\prime) \iota^{-1}(\alpha_2)+ \cdots +(}&{}^{m_{r-1}^\prime-m_r^\prime)\iota^{-1}(\alpha_r)} \\ & x_0^{k_1-1} x_{\zeta_N^{m_1^\prime}} x_0^{k_2-1} x_{\zeta_N^{m_2^\prime}} \cdots x_0^{k_r-1} x_{\zeta_N^{m_r^\prime}} x_0^{k_{r+1}-1} \end{aligned}\\
    & \begin{aligned}= \sum_{m_1=1}^N \sum_{m_2=1}^N \cdots \sum_{m_r=1}^N \zeta_N^{-m_1 \iota^{-1}(\alpha_1) - m_2 \iota^{-1}(\alpha_2)}&{}^{- \cdots - m_r \iota^{-1}(\alpha_r)} \\ & x_0^{k_1-1} x_{\zeta_N^{m_1}} x_0^{k_2-1} x_{\zeta_N^{m_1+m_2}} \cdots x_0^{k_r-1} x_{\zeta_N^{m_1 + \cdots + m_r}} x_0^{k_{r+1}-1}, \end{aligned}
    \end{align*}
    where the last equality follows from the change of variable $m_1 = m_1^\prime$, $m_2 = m_2^\prime - m_1^\prime, \dots$, $m_r = m_r^\prime - m_{r-1}^\prime$. On the other hand,
    \begin{align*}
    & \mathbf{p} \circ {\mathcal F}(\tilde{x}^{k_1-1}\tilde{x}_{\alpha_1} \tilde{x}^{k_2-1}\tilde{x}_{\alpha_2} \cdots \tilde{x}^{k_r-1}\tilde{x}_{\alpha_r} \tilde{x}^{k_{r+1}-1}) \\
    & = \mbox{\small$\displaystyle\mathbf{p}\left(x_0^{k_1-1} \sum_{m_1=1}^N \zeta_N^{-m_1 \iota^{-1}(\alpha_1)} x_{\zeta_N^{m_1}} x_0^{k_2-1} \sum_{m_2=1}^N \zeta_N^{-m_2 \iota^{-1}(\alpha_2)} x_{\zeta_N^{m_2}} \cdots x_0^{k_r-1} \sum_{m_r=1}^N \zeta_N^{-m_r \iota^{-1}(\alpha_r)} x_{\zeta_N^{m_r}} x_0^{k_{r+1}-1}\right)$} \\
    & \begin{aligned}= \sum_{m_1=1}^N \sum_{m_2=1}^N \cdots \sum_{m_r=1}^N \zeta_N^{-m_1 \iota^{-1}(\alpha_1) - m_2 \iota^{-1}(\alpha_2)}&{}^{- \cdots - m_r \iota^{-1}(\alpha_r)} \\ & x_0^{k_1-1} x_{\zeta_N^{m_1}} x_0^{k_2-1} x_{\zeta_N^{m_1+m_2}} \cdots x_0^{k_r-1} x_{\zeta_N^{m_1 + \cdots + m_r}} x_0^{k_{r+1}-1}, \end{aligned}
    \end{align*}
    thus proving that ${\mathcal F} \circ \widetilde{\mathbf{q}} = \mathbf{p} \circ {\mathcal F}$.
\end{proof}

\begin{lem}
    \begin{enumerate}[label=(\alph*), leftmargin=*]
        \item \label{lem: Hopf compatibility shuffle} The map ${\mathcal F} : (\Q(\mu_N)\langle\langle\widetilde{X}\rangle\rangle, \widehat{\Delta}_{\tilde{\sh}}) \to (\Q(\mu_N)\langle\langle X\rangle\rangle, \widehat{\Delta}_\sh)$ is a Hopf algebra isomorphism;
        \item \label{lem: Hopf compatibility stuffle} The map ${\mathcal F}_Y : (\Q(\mu_N)\langle\langle\widetilde{Y}\rangle\rangle, \widehat{\Delta}_{\tilde{\ast}}) \to (\Q(\mu_N)\langle\langle Y\rangle\rangle, \widehat{\Delta}_\ast)$ is Hopf algebra isomorphism.
    \end{enumerate}
    \label{lem: Hopf compatibility}
\end{lem}
\begin{proof}
    \begin{enumerate}[label=(\alph*), leftmargin=*]
        \item The statement is equivalent to the commutativity of the following diagram
        \[\begin{tikzcd}
            \Q(\mu_N)\langle\langle\widetilde{X}\rangle\rangle \ar[rr, "{\mathcal F}"] \ar[d, "\widehat{\Delta}_{\tilde{\sh}}"'] && \Q(\mu_N)\langle\langle X\rangle\rangle \ar[d, "\widehat{\Delta}_\sh"] \\
            \Q(\mu_N)\langle\langle \widetilde{X}\rangle\rangle^{\hat{\otimes} 2} \ar[rr, "{\mathcal F}^{\otimes 2}"] && \Q(\mu_N)\langle\langle X\rangle\rangle^{\hat{\otimes} 2}
        \end{tikzcd}\]
        Since all arrows are $\Q(\mu_N)$-algebra morphisms, it suffices to verify equality of the morphisms $\widehat{\Delta}_\sh \circ {\mathcal F}$ and ${\mathcal F}^{\otimes 2} \circ \widehat{\Delta}_{\tilde{\sh}}$ on generators of the algebra $\Q(\mu_N)\langle\langle\widetilde{X}\rangle\rangle$. This can be achieved through an immediate computation.
        \item The statement is equivalent to the commutativity of the following diagram
        \[\begin{tikzcd}
            \Q(\mu_N)\langle\langle\widetilde{Y}\rangle\rangle \ar[rr, "{\mathcal F}_Y"] \ar[d, "\widehat{\Delta}_{\tilde{\ast}}"'] && \Q(\mu_N)\langle\langle Y\rangle\rangle \ar[d, "\widehat{\Delta}_\ast"] \\
            \Q(\mu_N)\langle\langle \widetilde{Y}\rangle\rangle^{\hat{\otimes} 2} \ar[rr, "{\mathcal F}_Y^{\otimes 2}"] && \Q(\mu_N)\langle\langle Y\rangle\rangle^{\hat{\otimes} 2}
        \end{tikzcd}\]
        Since all arrows are $\Q(\mu_N)$-algebra morphisms, it suffices to verify equality of the morphisms $\widehat{\Delta}_\ast \circ {\mathcal F}_Y$ and ${\mathcal F}_Y^{\otimes 2} \circ \widehat{\Delta}_{\tilde{\ast}}$ on generators of the algebra $\Q(\mu_N)\langle\langle\widetilde{X}\rangle\rangle$. This can be achieved through an immediate computation.
    \end{enumerate}
\end{proof}

\noindent For $a \in \llbracket 1, N\rrbracket$, we define the $\Q(\mu_N)$-algebra automorphism $\tilde{t}_a$ of $\Q(\mu_N)\langle\langle\widetilde{X}\rangle\rangle$ as follows:
\[
    \tilde{x} \mapsto \tilde{x} \text{ and } \tilde{x}_\alpha \mapsto \zeta_N^{a \iota^{-1}(\alpha)} \tilde{x}_\alpha \text{ for } \alpha \in \Z / N \Z.   
\]
\begin{lem}
    \label{tzetaa_ta}
    For $a \in \llbracket 1, N\rrbracket$, the following diagram of $\Q(\mu_N)$-algebra morphisms is commutative:
\[\begin{tikzcd}
    \Q(\mu_N)\langle\langle \widetilde{X}\rangle\rangle \ar[rr, "{\mathcal F}"] \ar[d, "\tilde{t}_a"'] &&  \ar[d, "t_{\zeta_N^a}"] \Q(\mu_N)\langle\langle X\rangle\rangle \\
    \Q(\mu_N)\langle\langle \widetilde{X}\rangle\rangle \ar[rr, "{\mathcal F}"] && \Q(\mu_N)\langle\langle X\rangle\rangle
\end{tikzcd}\]
where  $t_{\zeta_N^a}$ is the automorphism of $\Q(\mu_N)\langle\langle X\rangle\rangle$ given in \eqref{eq: auto t}.
\end{lem}

\begin{proof}
Since all the maps appearing in the diagram are algebra morphisms, it is enough to show the claim for algebraic generators  $\tilde x$ and $\tilde x_\alpha$ ($\alpha \in \Z/N\Z$).
We have
\begin{align*}
{\mathcal F}\circ\tilde{t}_a(\tilde{x}_\alpha) & = {\mathcal F}(\zeta_N^{a\iota^{-1}(\alpha)}\tilde x_\alpha) =
\zeta_N^{a\iota^{-1}(\alpha)} \sum_{m=1}^N \zeta_N^{-m\iota^{-1}(\alpha)}x_{\zeta_N^m}, \\
t_{\zeta_N^a} \circ {\mathcal F}(\tilde x_b) &
= t_{\zeta_N^a} \left( \sum_{m^\prime=1}^N\zeta_N^{-m^\prime\iota^{-1}(\alpha)} x_{\zeta_N^{m^\prime}} \right)
= \sum_{m^\prime=1}^N\zeta_N^{-m^\prime\iota^{-1}(\alpha)} x_{\zeta_N^{m^\prime + a}} \\
& = \sum_{m=1}^N \zeta_N^{-(m-a)\iota^{-1}(\alpha)} x_{\zeta_N^m}
= \zeta_N^{a \iota^{-1}(\alpha)} \sum_{m=1}^N \zeta_N^{-m \iota^{-1}(\alpha)} x_{\zeta_N^m}.
\end{align*}
We also have ${\mathcal F}\circ\tilde{t}_a(\tilde{x})=t_{\zeta_N^a}\circ{\mathcal F}(\tilde x)$.
Whence our claim is obtained.
\end{proof}

\subsection{Construction of a Lie bracket \texorpdfstring{$\langle-,-\rangle^{\widetilde{}}$}{<-,->~}}
From now on, denote by $G_N$ the Galois group $\mathrm{Gal}(\mathbb{Q}(\mu_N)/\mathbb{Q})$. Recall that
\[
    G_N = \{ \sigma \in \mathrm{Aut}_\Q(\Q(\mu_N)) \, \mid \, \exists k \in \llbracket 1, N-1 \rrbracket \text{ with } \gcd(k, N)=1 \text{ such that } \sigma(\zeta_N)=\zeta_N^k \}.
\]
Therefore, one may identify $G_N$ with $(\mathbb{Z}/N\mathbb{Z})^\times$ under the natural isomorphism
\begin{equation}\label{eq:natural identification}
    G_N \simeq (\mathbb{Z}/N\mathbb{Z})^\times    
\end{equation}
given by $\sigma \mapsto \iota(k)$, for some $k \in \llbracket 1, N-1 \rrbracket$ with $\gcd(k, N)=1$.

\begin{defn}
    For $a \in \llbracket 1, N\rrbracket$, we define the $\Q(\mu_N)$-linear automorphism of $\Q(\mu_N)\langle\langle\widetilde{X}\rangle\rangle$ as follows:
    \[
        \widetilde{T}_a := \frac{1}{N} \sum_{m=1}^N \zeta_N^{-ma} \ \tilde{t}_m.
    \]    
\end{defn}

\begin{lem}\label{lem: Q-structure for T}
    The operator $\widetilde{T}_a$ ($a \in \llbracket 1, N\rrbracket$)
    preserves the $\Q$-form,
    that is, $\widetilde{T}_a \in \mathrm{Aut}_{\Q{\text{-}\mathrm{lin}}}(\Q\langle\langle\widetilde{X}\rangle\rangle)$.
\end{lem}

\begin{proof}
    The group $G_N$ % $\mathrm{Gal}(\Q(\mu_N)/\Q)$ 
    acts on $\Q(\mu_N)\left[(\tilde{t}_a)_{a \in \llbracket 1, N\rrbracket}\right]$ by
    \[
        \sigma \cdot r = \sigma(r), \text{ for } r \in \Q(\mu_N) \text{ and } \sigma \cdot \tilde{t}_a = \tilde{t}^{k}_a \text{ for } a \in \llbracket 1, N\rrbracket.   
    \]
    Next, recall that for any $k \in \llbracket 1, N-1 \rrbracket$ with $\gcd(k, N)=1$, we have 
    \begin{equation}
        \label{tk_identity}
        \tilde{t}_a^k = \tilde{t}_{\iota^{-1}(\overline{a k})}.
    \end{equation}
    Indeed, we have $\tilde{t}_a^k (\tilde x) = \tilde x = \tilde t_{\iota^{-1}(\overline{a k})}(\tilde x)$ and for $\alpha \in \Z / N \Z$
    \[
        \tilde{t}_a^k (\tilde x_\alpha) 
        %= \tilde{t}_a^k (\tilde x_\alpha) 
        = \zeta_N^{a k \iota^{-1}(\alpha)} \tilde{x}_\alpha = \zeta_N^{\iota^{-1}(\overline{a k}) \iota^{-1}(\alpha)} \tilde{x}_\alpha = \tilde{t}_{\iota^{-1}(\overline{a k})}(\tilde{x}_\alpha).
    \]
    Finally, for any $\widetilde{\psi} \in \Q\langle\langle\widetilde{X}\rangle\rangle$, we have
    \begin{align*}
        \sigma \cdot \widetilde{T}_a (\widetilde{\psi}) & = \frac{1}{N} \sum_{m^\prime=1}^N \sigma(\zeta_N^{-m^\prime a}) \ \sigma \cdot \tilde{t}_{m^\prime}(\widetilde{\psi}) = \frac{1}{N} \sum_{m^\prime=1}^N \zeta_N^{-k m^\prime a} \ \tilde{t}_{m^\prime}^k(\widetilde{\psi}) = \frac{1}{N} \sum_{m^\prime=1}^N \zeta_N^{-\iota^{-1}(\overline{k m^\prime}) a} \ \tilde{t}_{\iota^{-1}(\overline{k m^\prime})}(\widetilde{\psi}) \\
        & = \frac{1}{N} \sum_{m=1}^N \zeta_N^{-m a} \ \tilde{t}_m(\widetilde{\psi}) = \widetilde{T}_a (\widetilde{\psi}),
    \end{align*}
    where the third equality comes from identity \eqref{tk_identity}. Hence $\widetilde{T}_a(\widetilde{\psi}) \in \Q\langle\langle\widetilde{X}\rangle\rangle$.
\end{proof}

\noindent Lemma \ref{lem: Q-structure for T} enables us to define the following:
\begin{defn}
    For $\widetilde{\psi} \in \Q\langle\langle \widetilde{X}\rangle\rangle$, define a derivation $\tilde{d}_{\widetilde{\psi}} : \Q\langle\langle \widetilde{X}\rangle\rangle \to \Q\langle\langle \widetilde{X}\rangle\rangle$ by
    \begin{align*}
        \tilde{x} \mapsto 0, \text{ and } \, \tilde{x}_\alpha \mapsto \widetilde{T}_{\iota^{-1}(\alpha)}\left(\Bigg[\sum_{\beta \in \Z / N \Z} \tilde{x}_\beta, \ \widetilde{\psi}\Bigg]\right) \text{ for } \alpha \in \Z / N \Z,
    \end{align*}
\end{defn}

\begin{prop}\label{prop: commutativity for d}
For $\widetilde{\psi} \in \Q\langle\langle \widetilde{X}\rangle\rangle$, the following diagram is commutative\footnote{Here, we abusively denote by $\tilde{d}_{\widetilde{\psi}}$ (resp. $d_{{\mathcal F}(\widetilde{\psi})}$) the derivation $\mathrm{id}_{\Q(\mu_N)} \otimes \tilde{d}_{\widetilde{\psi}}$ (resp. $\mathrm{id}_{\Q(\mu_N)} \otimes d_{{\mathcal F}(\widetilde{\psi})}$) of $\Q(\mu_N)\langle\langle\widetilde{X}\rangle\rangle \simeq \Q(\mu_N) \hat{\otimes}_\Q \Q\langle\langle\widetilde{X}\rangle\rangle$ (resp. $\Q(\mu_N)\langle\langle X\rangle\rangle \simeq \Q(\mu_N) \hat{\otimes}_\Q \Q\langle\langle X\rangle\rangle$).}
\[\begin{tikzcd}
    \Q(\mu_N)\langle\langle\widetilde{X}\rangle\rangle \ar[rr, "{\mathcal F}"] \ar[d, "\tilde{d}_{\widetilde{\psi}}"'] && \Q(\mu_N)\langle\langle X\rangle\rangle \ar[d, "d_{{\mathcal F}(\widetilde{\psi})}"] \\ 
    \Q(\mu_N)\langle\langle\widetilde{X}\rangle\rangle \ar[rr, "{\mathcal F}"] && \Q(\mu_N)\langle\langle X\rangle\rangle
\end{tikzcd}\]
where $d_{\mathcal{F}(\widetilde{\psi})}$ is the derivation in \eqref{eq:derivation d}.
\end{prop}

\begin{proof}
Since all vertical arrows of the diagram are derivations and all horizontal arrows are algebra morphisms, we will demonstrate this equality by applying both sides to generators $\tilde{x}$ and $\tilde{x}_\alpha$ and showing that they yield the same result. The computation for $\tilde{x}$ being immediate, we will focus on $\tilde{x}_\alpha$. We have  
\begin{align*}
    {\mathcal F} \circ \tilde{d}_{\widetilde{\psi}}(\tilde{x}_a) & = {\mathcal F} \circ \widetilde{T}_{\iota^{-1}(\alpha)}\left(\Bigg[\sum_{\beta \in \Z / N \Z} \tilde{x}_\beta, \ \widetilde{\psi}\Bigg]\right) = \frac{1}{N} \sum_{m=1}^N \zeta_N^{-m\iota^{-1}(\alpha)} \ {\mathcal F} \circ \tilde{t}_m \left(\Bigg[\sum_{\beta \in \Z / N \Z} \tilde{x}_\beta, \ \widetilde{\psi}\Bigg]\right) \\
    & = \sum_{m=1}^N \zeta_N^{-m\iota^{-1}(\alpha)} \ t_{\zeta_N^m} \circ {\mathcal F} \left(\Bigg[\frac{1}{N}\sum_{\beta \in \Z / N \Z} \tilde{x}_\beta, \ \widetilde{\psi}\Bigg]\right) \\
    & = \sum_{m=1}^N \zeta_N^{-m\iota^{-1}(\alpha)} \ t_{\zeta_N^m} \left(\Bigg[{\mathcal F}\left(\frac{1}{N}\sum_{\beta \in \Z / N \Z} \tilde{x}_\beta\right), \ {\mathcal F}(\widetilde{\psi})\Bigg]\right) \\
    & = \sum_{m=1}^N \zeta_N^{-m\iota^{-1}(\alpha)} \ t_{\zeta_N^m} \left(\left[x_1, \ {\mathcal F}(\widetilde{\psi})\right]\right) = \sum_{m=1}^N \zeta_N^{-m\iota^{-1}(\alpha)} \ \left[t_{\zeta_N^m}(x_1), \ t_{\zeta_N^m} \circ {\mathcal F}(\widetilde{\psi})\right] \\
    & = \sum_{m=1}^N \zeta_N^{-m\iota^{-1}(\alpha)} \ \left[x_{\zeta_N^m}, \ t_{\zeta_N^m} \left({\mathcal F}(\widetilde{\psi})\right)\right] = \sum_{m=1}^N \zeta_N^{-m\iota^{-1}(\alpha)} \ d_{{\mathcal F}(\widetilde{\psi})}(x_{\zeta_N^m}) \\
    & = d_{{\mathcal F}(\widetilde{\psi})}\left(\sum_{m=1}^N \zeta_N^{-m\iota^{-1}(\alpha)} \ x_{\zeta_N^m}\right)  = d_{{\mathcal F}(\widetilde{\psi})} \circ {\mathcal F}(\tilde{x}_\alpha),
\end{align*}
where the third equality comes from Lemma \ref{tzetaa_ta}.
\end{proof}

\begin{defn}\label{defn:Congruent double shuffle Lie algebra}
    The \emph{$N$-congruent double shuffle Lie algebra}
    $\dmr_0^{[N]}$ is the $\Q$-linear space consisting of elements $\widetilde{\psi} \in \Q\langle\langle \widetilde{X}\rangle\rangle$ such that
    \begin{enumerate}[label=(\roman*), leftmargin=*, itemsep=3mm]
        \begin{multicols}{2}
        \item \label{dmrN_i} $\displaystyle(\widetilde{\psi} \mid \tilde{x}) = \sum_{\alpha \in \Z / N \Z} (\widetilde{\psi} \mid \tilde{x}_\alpha) = 0$;
        \item \label{dmrN_ii} $\widehat{\Delta}_{\tilde{\sh}}(\widetilde{\psi}) = \widetilde{\psi} \otimes 1 + 1 \otimes \widetilde{\psi}$;
        \item \label{dmrN_iii} $\left(\widetilde{\psi} \mid x_\alpha - x_{-\alpha}\right) = 0$ for any $\alpha \in \Z / N \Z$;
        \item \label{dmrN_iv} $\widehat{\Delta}_{\tilde{\ast}}(\widetilde{\psi}_{\tilde{\ast}}) = \widetilde{\psi}_{\tilde{\ast}} \otimes 1 + 1 \otimes \widetilde{\psi}_{\tilde{\ast}}$;
        \end{multicols}
    \end{enumerate}
    where
    \[
        \widetilde{\psi}_{\tilde{\ast}} = \pi_{\widetilde{Y}} \circ \widetilde{\mathbf{q}}^{-1}(\widetilde{\psi}) + \sum_{n \geq 2} \sum_{a=1}^N \sum_{b_1=1}^N \cdots \sum_{b_n=1}^N \frac{(-1)^{n-1}}{n \ N^{n+1}} \left(\widetilde{\psi} \mid \tilde{x}^{n-1} \tilde{x}_{\iota(a)}\right) \  \tilde{y}_{1,\iota(b_1)} \cdots \tilde{y}_{1,\iota(b_n)} \in \Q\langle\langle \widetilde{Y}\rangle\rangle.
    \]
\end{defn}

Similarly to \eqref{eq: bracket for dmr}, we define a Lie bracket on $\Q\langle\langle \widetilde{X}\rangle\rangle$ as follows:
\begin{equation}\label{eq: bracket for dmr[N]}
    \langle \widetilde{\psi}_1, \widetilde{\psi}_2 \rangle^{\tilde{}} := \tilde{d}_{\widetilde{\psi}_1}(\widetilde{\psi}_2) - \tilde{d}_{\widetilde{\psi}_2}(\widetilde{\psi}_1) + [\widetilde{\psi}_1,\widetilde{\psi}_2],
\end{equation}
for any $\widetilde{\psi}_1, \widetilde{\psi}_2 \in \Q\langle\langle \widetilde{X}\rangle\rangle$.

\begin{thm}\label{main theorem}
    \begin{enumerate}[label=(\alph*), leftmargin=*]
        \item \label{LA_dmrN} The $\Q$-linear space $\dmr_0^{[N]}$  forms a Lie algebra under the bracket $\langle-,-\rangle^{\widetilde{}}$.
        \item \label{QzetaN_isom} The map $\mathcal{F}$ induces a $\Q(\mu_N)$-Lie algebra isomorphism
        \begin{equation}\label{eq: QzetaN isom}
            \left(\Q(\mu_N) \ \hat{\otimes}_\Q \ \dmr_0^{[N]}, \ \langle-,-\rangle^{\widetilde{}}\right) \xrightarrow{\simeq} \left(\Q(\mu_N) \ \hat{\otimes}_\Q \ \dmr_0^{\mu_N},\  \langle-,-\rangle\right).
        \end{equation}
    \end{enumerate}
\end{thm}

\begin{proof}
    The claim \ref{LA_dmrN} follows from the claim \ref{QzetaN_isom} since for any $\widetilde{\psi} \in \Q\langle\langle\widetilde{X}\rangle\rangle$, the map $\tilde{d}_{\widetilde{\psi}}$ is, by definition, a derivation of $\Q\langle\langle\widetilde{X}\rangle\rangle\subset
    \Q(\mu_N)\langle\langle\widetilde{X}\rangle\rangle$.
    We now proceed to prove claim \ref{QzetaN_isom}. \newline
    One immediately checks that condition \ref{dmrN_i} of Definition \ref{defn:Congruent double shuffle Lie algebra} corresponds to condition \ref{dmrmuN_i} of Definition \ref{dmrmuN} under the isomorphism ${\mathcal F}$. In addition, thanks to Lemma \ref{lem: Hopf compatibility}.\ref{lem: Hopf compatibility shuffle}, condition \ref{dmrN_ii} of Definition \ref{defn:Congruent double shuffle Lie algebra} corresponds to condition \ref{dmrmuN_ii} of Definition \ref{dmrmuN} under the isomorphism ${\mathcal F}$. \newline
    The equivalence between condition \ref{dmrN_iii} of Definition \ref{defn:Congruent double shuffle Lie algebra} and condition \ref{dmrmuN_iii} of Definition \ref{dmrmuN} is established through the following identities
    \[
        x_{\zeta_N^m} - x_{\zeta_N^{-m}} = \frac{1}{N} \sum_{a=1}^N \zeta_N^{am} {\mathcal F}(\tilde x_{\iota(a)}-\tilde x_{-\iota(a)})
    \]
    and
    \[
        \tilde x_\alpha - \tilde x_{-\alpha} = \sum_{m=1}^N \zeta_N^{-m\iota^{-1}(\alpha)} {\mathcal F}^{-1}(x_{\zeta_N^{m}}- x_{\zeta_N^{-m}}),
    \]
    for $m \in \llbracket 1, N\rrbracket$ and $\alpha \in \Z / N \Z$. The equivalence between condition \ref{dmrN_iv} of Definition \ref{defn:Congruent double shuffle Lie algebra} and condition \ref{dmrmuN_iv} of Definition \ref{dmrmuN} follows from the subsequent arguments. Let $\psi = {\mathcal F}(\widetilde\psi)$ with $\widetilde\psi \in \Q\langle\langle\widetilde{X}\rangle\rangle$. Then we have
    \begin{align*}
        {\mathcal F}_Y(\widetilde\psi_\ast) & ={\mathcal F}_Y\Bigg(\pi_{\widetilde{Y}} \circ \widetilde{\mathbf{q}}^{-1}(\widetilde{\psi}) + \sum_{n \geq 2} \sum_{a=1}^N \sum_{b_1=1}^N \cdots \sum_{b_n=1}^N \frac{(-1)^{n-1}}{n \ N^{n+1}} \big(\widetilde{\psi} \mid \tilde{x}^{n-1} \tilde{x}_{\iota(a)}\big) \  \tilde{y}_{1,\iota(b_1)} \cdots \tilde{y}_{1,\iota(b_n)}\Bigg) \\
        & = \pi_Y \circ \mathbf{p}^{-1}(\psi) + \sum_{n \geq 2} \frac{(-1)^{n-1}}{n} \Big({\widetilde\psi} \bigm| \frac{1}{N}\sum_{a=1}^N\tilde{x}^{n-1} \tilde{x}_{\iota(a)}\Big) y_{1,1}^n \\
        & = \pi_Y \circ \mathbf{p}^{-1}(\psi) + \sum_{n \geq 2} \frac{(-1)^{n-1}}{n} \left(\tilde{\psi} \mid {\mathcal F}^{-1}({x}_0^{n-1} {x}_1)\right) y_{1,1}^n \\
        & = \pi_Y \circ \mathbf{p}^{-1}(\psi) + \sum_{n \geq 2}\frac{(-1)^{n-1}}{n} \left({\mathcal F}({\psi}) \mid {x}_0^{n-1} {x}_1\right) y_{1,1}^n = \psi_\ast,
    \end{align*}
    where the second equality follows from Lemma \ref{lem:diagpi_diagqp}. By Lemma \ref{lem: Hopf compatibility} \ref{lem: Hopf compatibility stuffle}, we establish the wanted equivalence under the map ${\mathcal F}_Y$. \newline
    The compatibility of Lie brackets is verified as follows: set $\psi_1 = {\mathcal F}(\widetilde{\psi}_1)$ and $\psi_2 = {\mathcal F}(\widetilde{\psi}_2)$ with $\widetilde{\psi}_1, \widetilde{\psi}_2 \in \dmr_0^{[N]}$.
    Then by Proposition \ref{prop: commutativity for d}, we have
    \begin{align*}
        \langle \widetilde{\psi}_1, \widetilde{\psi}_2 \rangle^{\tilde{}} 
        & = \tilde{d}_{\widetilde{\psi}_1}(\widetilde{\psi}_2) - \tilde{d}_{\widetilde{\psi}_2}(\widetilde{\psi}_1) + [\widetilde{\psi}_1,\widetilde{\psi}_2] \\
        & = {\mathcal F}^{-1}({d}_{{\mathcal F}(\widetilde{\psi}_1)}({\mathcal F}(\widetilde{\psi}_2))) - {\mathcal F}^{-1}({d}_{{\mathcal F}(\widetilde{\psi}_2)}({\mathcal F}(\widetilde{\psi}_1))) + {\mathcal F}^{-1}([\psi_1,\psi_2]) \\
        & = {\mathcal F}^{-1}\left({d}_{{\psi_1}}({\psi_2}) - {d}_{{\psi_2}}({\psi_1}) + [{\psi_1},{\psi_2}]\right) \\
        & = {\mathcal F}^{-1}\left( \langle {\psi_1}, {\psi_2} \rangle\right) = {\mathcal F}^{-1}( \langle {{\mathcal F}(\tilde \psi_1)}, {{\mathcal F}(\tilde \psi_2)} \rangle).
    \end{align*}
\end{proof}

\begin{rem}
Building upon the work of \cite{BT17, YZ15}, Kanno \cite{K24} investigated double shuffle relations for level $N$ multiple Eisenstein series, 
with $N$-CMZVs as Fourier expansion constant terms. His explicit presentation of Goncharov coproduct \cite[Proposition 4.15]{K24} may offer 
an alternative view on the Lie bracket formula \eqref{eq: bracket for dmr[N]}. We anticipate that the Lie algebra $\dmr_0^{[N]}$ will contribute  to the study of algebraic structure of level $N$ multiple Eisenstein series.
\end{rem}

\subsection{\texorpdfstring{$G_N$}{GN}-invariance characterization theorem}
\noindent For each $\gamma \in (\mathbb{Z}/N\mathbb{Z})^\times$, we define $\delta_\gamma$ to be the algebra automorphism of $\Q\langle\langle X\rangle\rangle$ given by
\begin{align}\label{eq: auto delta}
    x_0 \mapsto x_0, \text{ and } \, x_{\zeta} \mapsto x_{\zeta^{\iota^{-1}(\gamma)}}, \text{ for } \zeta \in \mu_N.
\end{align}
We also define $\widetilde{\delta}_\gamma$ to be the algebra automorphism of $\mathbb{Q}\langle\langle\widetilde{X}\rangle\rangle$ given by
\begin{align}\label{eq: auto tildedelta}
        \tilde{x} \mapsto \tilde{x}, \text{ and } \, \tilde{x}_{\alpha} \mapsto \tilde{x}_{\gamma\alpha}, \text{ for } \alpha \in \Z/N\Z.
\end{align}
These correspondences induce actions of $(\mathbb{Z}/N\mathbb{Z})^\times$ on the algebras $\mathbb{Q}\langle\langle{X}\rangle\rangle$ and $\mathbb{Q}\langle\langle\widetilde{X}\rangle\rangle$ by $\gamma \mapsto \delta_\gamma$ and $\gamma \mapsto \widetilde{\delta}_\gamma$ respectively. We denote by $\sigma \mapsto \Delta_\sigma$ (resp. $\sigma \mapsto \widetilde{\Delta}_\sigma$) the action of $G_N$ on $\Q(\mu_N)\langle\langle{X}\rangle\rangle \simeq \Q(\mu_N) \hat{\otimes} \Q\langle\langle{X}\rangle\rangle$ (resp. on $\Q(\mu_N)\langle\langle\widetilde{X}\rangle\rangle \simeq \Q(\mu_N) \hat{\otimes} \Q\langle\langle\widetilde{X}\rangle\rangle$)
where $\sigma$ acts on the first component via the Galois action and on the second component via the $(\Z/N\Z)^\times$-action given in \eqref{eq: auto delta} (resp. \eqref{eq: auto tildedelta}) thanks to the isomorphism \eqref{eq:natural identification}.

\noindent We now consider the linear map $\mathcal{F} : \Q(\mu_N)\langle\langle\widetilde{X}\rangle\rangle \to \Q(\mu_N)\langle\langle{X}\rangle\rangle$ defined in Lemma \ref{lem:iso_mathcalF}, and establish the following statement which expresses equivariance properties of $\mathcal{F}$ with respect to Galois actions:
\begin{lem}\label{lem:galois_action_compatibility}
For $\sigma \in G_N$, the following equalities hold:
\begin{equation}\label{eq:mathcalF and sigma 1}
    \mathcal{F} \circ (\sigma \otimes \mathrm{id}_{\Q\langle\langle\widetilde{X}\rangle\rangle}) = \Delta_\sigma \circ \mathcal{F},
\end{equation}
and
\begin{equation}\label{eq:mathcalF and sigma 2}
    \mathcal{F} \circ \widetilde{\Delta}_\sigma = (\sigma \otimes \mathrm{id}_{\Q\langle\langle{X}\rangle\rangle}) \circ \mathcal{F},
\end{equation}
where $\sigma \otimes \mathrm{id}_{\Q\langle\langle{X}\rangle\rangle}$ (resp. $\sigma \otimes \mathrm{id}_{\Q\langle\langle\widetilde{X}\rangle\rangle}$) denotes the action on $\Q(\mu_N)\langle\langle{X}\rangle\rangle \simeq \Q(\mu_N) \hat{\otimes} \Q\langle\langle{X}\rangle\rangle$ (resp. $\Q(\mu_N)\langle\langle\widetilde{X}\rangle\rangle \simeq \Q(\mu_N) \hat{\otimes} \Q\langle\langle\widetilde{X}\rangle\rangle$) such that $\sigma$ acts only on the first component via the Galois action.
\end{lem}
\begin{proof}
    We will prove equality \eqref{eq:mathcalF and sigma 2}. Equality \eqref{eq:mathcalF and sigma 1} can be proven in a similar manner. \linebreak
    First, equality \eqref{eq:mathcalF and sigma 2} is immediate for $\tilde{x}$. Next, let  $z = r \tilde{x}_\alpha$ with $r \in \Q(\mu_N)$ and $\alpha \in \Z/N\Z$. Let $\gamma$ be an element in $ (\Z/N\Z)^\times$ which corresponds to $\sigma$ under the identification \eqref{eq:natural identification}.
    Then 
    \begin{align*}        
        \mathcal{F}((\widetilde{\Delta}_\sigma(z)) & = \mathcal{F}((\sigma \otimes \widetilde{\delta}_\sigma)(r \tilde{x}_\alpha))
        = \mathcal{F}(\sigma(r) \tilde{x}_{\gamma \alpha}) = \sigma(r) \mathcal{F}(\tilde{x}_{\gamma \alpha}) = \sigma(r) \sum_{m=1}^N \zeta_N^{-m \iota^{-1}(\gamma \alpha)} x_{\zeta_N^m} \\
        & = \sigma(r) \sum_{m=1}^N \sigma(\zeta_N^{-m\iota^{-1}(\alpha)}) x_{\zeta_N^m} = (\sigma \otimes \mathrm{id}) \left(r \sum_{m=1}^N \zeta_N^{-m \iota^{-1}(\alpha)} x_{\zeta_N^m}\right)
        %=(\sigma\otimes\mathrm{id})(r{\mathcal F}(\tilde{x}_{\alpha}))
        = (\sigma \otimes \mathrm{id}) (\mathcal{F}(r \tilde{x}_{\alpha})) \\
        & = (\sigma \otimes \mathrm{id})(\mathcal{F}(z)).
    \end{align*}
    Finally, since $\mathcal{F}$ is an algebra morphism, equality \eqref{eq:mathcalF and sigma 2} follows for all elements in $\Q(\mu_N)\langle\langle\widetilde{X}\rangle\rangle$.
\end{proof}

\begin{prop} \label{prop:dmr_stability}
The subspace $\dmr_0^{[N]}$  (resp. $\dmr_0^{\mu_N}$)
is stable under the action of $(\Z/N\Z)^\times$ on 
$\Q\langle\langle\widetilde{X}\rangle\rangle$
(resp. $\Q\langle\langle X\rangle\rangle$).
\end{prop}

\begin{proof}
    The invariance of these subspaces follows directly from an examination of their defining equations (i)-(iv). One can verify that each of these equations is preserved under the respective actions, thus ensuring the stability of the subspaces.
\end{proof}

\begin{cor}
    \label{cor: action of GN on QmuNxdmr0}
    The subspace $\Q(\mu_N) \ \hat{\otimes}_\Q \ \dmr_0^{[N]}$  (resp. $\Q(\mu_N) \ \hat{\otimes}_\Q \ \dmr_0^{\mu_N}$) is stable under the action $\sigma \mapsto \widetilde{\Delta}_\sigma$ (resp. $\sigma \mapsto \Delta_\sigma$) of $G_N$ on $\Q(\mu_N)\langle\langle\widetilde{X}\rangle\rangle$ (resp. $\Q(\mu_N)\langle\langle X\rangle\rangle$).
\end{cor}
\begin{proof}
    This follows immediately from Proposition \ref{prop:dmr_stability} and by definition of the action $\sigma \mapsto \widetilde{\Delta}_\sigma$ (resp. $\sigma \mapsto \Delta_\sigma$).  
\end{proof}

\noindent The following theorem establishes our $G_N$-invariance characterization result for the double shuffle Lie algebras:
\begin{thm}\label{Galois Descent theorem}
    Denote by $(\Q(\mu_N) \ \hat{\otimes}_\Q \ \dmr_0^{[N]})^{\widetilde{\Delta}_{G_N}}$ (resp. $(\Q(\mu_N) \ \hat{\otimes}_\Q \ \dmr_0^{\mu_N})^{\Delta_{G_N}}$) the invariant subspace under the action of $G_N$ described in Corollary \ref{cor: action of GN on QmuNxdmr0}. Under the isomorphism ${\mathcal F}$ given in \eqref{eq: QzetaN isom}, the following identities hold:
    \begin{equation}\label{eq:fixed_point_dmr_N}
        (\Q(\mu_N) \ \hat{\otimes}_\Q \ \dmr_0^{[N]})^{\widetilde{\Delta}_{G_N}} \simeq \dmr_0^{\mu_N}
    \end{equation}
    and
    \begin{equation}\label{eq:fixed_point_dmr_muN}
        \dmr_0^{[N]} \simeq (\Q(\mu_N) \ \hat{\otimes}_\Q \ \dmr_0^{\mu_N})^{\Delta_{G_N}}.
    \end{equation}
\end{thm}
\begin{proof}
    We will prove identity \eqref{eq:fixed_point_dmr_N}. Identity \eqref{eq:fixed_point_dmr_muN} can be proven in a similar manner. \linebreak Let $f \in (\Q(\mu_N) \ \hat{\otimes}_\Q \ \dmr_0^{[N]})^{\widetilde{\Delta}_{G_N}}$. By Theorem \ref{main theorem} (b), we have ${\mathcal F}(f)\in \Q(\mu_N) \ \hat{\otimes}_\Q \ \dmr_0^{\mu_N}$. Moreover, by equality \eqref{eq:mathcalF and sigma 2} of Lemma \ref{lem:galois_action_compatibility}, we have
    $$
    \mathcal{F}(f) = \mathcal{F}\left(\widetilde{\Delta}_\sigma(f)\right) = (\sigma \otimes \mathrm{id}_{\Q\langle\langle{X}\rangle\rangle}) \left(\mathcal{F}(f)\right).% = {\mathcal F}(\Delta(\sigma)(f))=(\sigma\otimes\mathrm{id}){\mathcal F}(f)
    $$
    Thus proving that ${\mathcal F}(f)\in\dmr_0^{\mu_N}$.
    Conversely, assume $g\in \dmr_0^{\mu_N}\subset \Q(\mu_N) \ \hat{\otimes}_\Q \ \dmr_0^{\mu_N}$. Then by Theorem \ref{main theorem} (b), ${\mathcal F}^{-1}(g)\in \Q(\mu_N) \ \hat{\otimes}_\Q \ \dmr_0^{[N]}$. From equality \eqref{eq:mathcalF and sigma 2} of Lemma \ref{lem:galois_action_compatibility}, we have
    $$
        \mathcal{F}^{-1}(g) = \mathcal{F}^{-1}\left((\sigma \otimes \mathrm{id}_{\Q\langle\langle{X}\rangle\rangle})(g)\right)
        = \widetilde{\Delta}_\sigma \circ \mathcal{F}^{-1}(g).
    $$
    Hence we have $\mathcal{F}^{-1}(g) \in (\Q(\mu_N) \ \hat{\otimes}_\Q \ \dmr_0^{[N]})^{\widetilde{\Delta}_{G_N}}$.
\end{proof}

%%%%%%%%%%%%%%%%%%%%%%%%%%%%%%%%%%%%%%%%%%%%%%%%%%%%%%%%%%%

\noindent Our investigation addresses a fundamental question concerning the descent of the isomorphism presented in equation \eqref{eq: QzetaN isom}. Specifically, we examine whether this isomorphism, initially defined over the cyclotomic field  $\Q(\mu_N)$, can be realized over the rational number field $\Q$. This inquiry leads us to the following problem:

\begin{prob}
    Are the Lie algebras $\dmr_0^{[N]}$ and $\dmr_0^{\mu_N}$ isomorphic over $\Q$ ?
\end{prob}

\noindent This problem might have potential implications for our understanding of the arithmetic properties of multiple zeta values. %and their relations to Galois theory. 
It may also provide insights into the deeper structure of motivic Lie algebras associated with mixed Tate motives over cyclotomic fields.

%%%%%%%%%%%%%%%%%%%%%%%%%%%%%%%%%%%%%%%%%%%%%%%%%%%%%%%%%%
\appendix
\section{Distribution relations on \texorpdfstring{$N$}{N}-CMZVs}
In this appendix, we aim to describe explicitly finite and regularized distribution relations among $N$-CMZVs by utilizing the finite and regularized distribution relations among $N$-MPVs.
Those latter are expressed by the following identity (see \cite{Gon98})
\begin{equation}
    \label{eq:dist_NMPV}
    \Li_{(k_1, \dots, k_r)}(\eta_1, \dots, \eta_r) = d^{k_1 + \cdots + k_r - r} \sum_{\substack{\zeta_i^d=\eta_i \\ 1 \leq i \leq r}} \Li_{(k_1, \dots, k_r)}(\zeta_1, \dots, \zeta_r),
\end{equation}
for any divisor $d$ of $N$, where $r, k_1, \dots, k_r$ are positive integers and $\eta_1, \dots, \eta_r$ are $\frac{N}{d}$\textsuperscript{th} roots of unity with $(k_1, \eta_1) \neq (1, 1)$.
In \cite[§2.5.3]{Rac02}, Racinet introduced an algebraic framework to describe these relations as follows: for $d$ a divisor of $N$, set $X_d := \{ x_0, x_\zeta \mid \zeta \in \mu_{\frac{N}{d}} \}$, 
{this is a subset of $X$ thanks to the equality $\mu_{\frac{N}{d}} = \mu_N^d$ of subgroups of $\mu_N$}. Define the algebra morphisms $p^d_\ast$ and $i_d^\ast : \Q\langle\langle{X}\rangle\rangle \to \Q\langle\langle{X_d}\rangle\rangle$ by (see also \cite[§4]{Zha10})
\[\begin{array}{ccc}
    p^d_\ast : \quad \begin{aligned}
        &x_0 \mapsto d \, x_0 \\
        &x_\zeta \mapsto x_{\zeta^d}
    \end{aligned}
    &
    \text{ and }
    &
    i_d^\ast : \quad \begin{aligned}
        &x_0 \mapsto x_0 \\
        &x_\zeta \mapsto \begin{cases} x_\zeta & \text{if } \zeta \in \mu_{\frac{N}{d}} \\ 0 & \text{otherwise}\end{cases}
    \end{aligned}
\end{array}\]

\begin{defn}[{\cite[(3.3.1.3)]{Rac02}}]
    Let $\dmrd_0^{\mu_N}$ be the $\Q$-Lie subalgebra of $\dmr_0^{\mu_N}$ defined by
    \[
        \dmrd_0^{\mu_N} := \Big\{ \psi \in \dmr_0^{\mu_N} \mid \forall d | N, \ p^d_\ast(\psi) = i_d^\ast(\psi) + \sum_{\zeta \in \mu_{\frac{N}{d}}} (\psi \mid x_\zeta) x_1 \Big\}.
    \]
\end{defn}

On the other hand, for a divisor $d$ of $N$, consider the substitution $\iota_d : \llbracket 1, \frac{N}{d}\rrbracket \to \Z/\frac{N}{d}\Z$, which associates to each element of $\llbracket 1, \frac{N}{d}\rrbracket$ its equivalence class. 
{One then checks that there exists a group isomorphism $\nu_d : \Z/\frac{N}{d}\Z \to d \Z/N\Z$ such that the following diagram}
\begin{equation}\label{diag:nu_d}
    \begin{tikzcd}
        & d \Z/N\Z \\
        \llbracket 1, \frac{N}{d} \rrbracket \ar[r, "\iota_d"] \ar[ru, "d \iota"] & \Z/\frac{N}{d}\Z \ar[u, "\nu_d"']
    \end{tikzcd}
\end{equation}
commutes, where $d \iota : m \mapsto d \iota(m)$, for any $m \in \llbracket 1, \frac{N}{d}\rrbracket$. \newline
Set $\widetilde{X}_d := \{ \tilde{x}, \tilde{x}_{\beta} \mid \beta \in \Z/\frac{N}{d}\Z \}$, which one sees as a subset of $X$ thanks to the isomorphism $\nu_d$.

\begin{defn}
    Define the algebra morphism $\mathcal{F}_d : \Q(\mu_N)\langle\langle{\widetilde{X}_d}\rangle\rangle \to \Q(\mu_N)\langle\langle{X_d}\rangle\rangle$ by
    \[
        \tilde{x} \mapsto x_0 \text{ and } \tilde{x}_{\beta} \mapsto \sum_{m=1}^{\frac{N}{d}} \zeta_N^{-m d \iota_d^{-1}(\beta)} x_{\zeta_N^{dm}} \text{ for } \beta \in \Z/\frac{N}{d}\Z. 
    \]
\end{defn}
\noindent By replacing $N$ by $\frac{N}{d}$ in $\mathcal{F}$, one checks that $\mathcal{F}_d$ is an isomorphism with reciprocal given by
\[
    x_0 \mapsto \tilde{x} \text{ and } x_{\zeta_N^{m d}} \mapsto \frac{d}{N} \sum_{a=1}^{\frac{N}{d}} \zeta_N^{a m d} \tilde{x}_{\iota_d(a)} \text{ for } m \in \llbracket 1, \frac{N}{d}\rrbracket. 
\]
\begin{defn}
    Define the algebra morphisms $\tilde{p}^d_\ast$ and $\tilde{i}_d^\ast : \Q\langle\langle\widetilde{X}\rangle\rangle \to \Q\langle\langle\widetilde{X}_d\rangle\rangle$ by
    \[\begin{array}{ccc}
        \tilde{p}^d_\ast : \quad \begin{aligned}
        & \tilde{x} && \mapsto d \, \tilde{x} \\
        & \tilde{x}_\alpha && \mapsto \begin{cases} d \, \tilde{x}_{\nu_d^{-1}(\alpha)} & \text{if } \alpha \in d \Z/N\Z \\ 0 & \text{otherwise}\end{cases}
        \end{aligned}
        &
        \text{ and }
        &
        \tilde{i}_d^\ast : \quad \begin{aligned}
        & \tilde{x} && \mapsto \tilde{x} \\
        & \tilde{x}_\alpha && \mapsto \tilde{x}_{\nu_d^{-1}(d \alpha)}
        \end{aligned}
    \end{array}\]
\end{defn}

\begin{lem} \label{lem:pi_tildepi}
    The following diagrams\footnote{Here, we abusively denote by $\tilde{p}_\ast^d$ and $p_\ast^d$ (resp. $\tilde{i}^\ast_d$ and $i^\ast_d$) the algebra morphisms $\mathrm{id}_{\Q(\mu_N)} \otimes \tilde{p}_\ast^d$ and $\mathrm{id}_{\Q(\mu_N)} \otimes p_\ast^d$ (resp. $\mathrm{id}_{\Q(\mu_N)} \otimes \tilde{i}^\ast_d$ and $\mathrm{id}_{\Q(\mu_N)} \otimes i^\ast_d$).}
    \[\begin{array}{lcr}
        \begin{tikzcd}
            \Q(\mu_N)\langle\langle\widetilde{X}\rangle\rangle \ar[rr, "{\mathcal F}"] \ar[d, "\tilde{p}^d_\ast"'] && \ar[d, "p^d_\ast"] \Q(\mu_N)\langle\langle{X}\rangle\rangle \\
            \Q(\mu_N)\langle\langle \widetilde{X}_d\rangle\rangle \ar[rr, "\mathcal{F}_d"] && \Q(\mu_N)\langle\langle{X}_d\rangle\rangle
        \end{tikzcd}
        &
        \text{ and }
        &
        \begin{tikzcd}
            \Q(\mu_N)\langle\langle\widetilde{X}\rangle\rangle \ar[rr, "{\mathcal F}"] \ar[d, "\tilde{i}_d^\ast"'] && \ar[d, "i_d^\ast"] \Q(\mu_N)\langle\langle{X}\rangle\rangle \\
            \Q(\mu_N)\langle\langle \widetilde{X}_d\rangle\rangle \ar[rr, "\mathcal{F}_d"] && \Q(\mu_N)\langle\langle{X}_d\rangle\rangle
        \end{tikzcd}
    \end{array}\]
    commute.
\end{lem}
\begin{proof}
    Let us show the commutativity of the left diagram. The proof is analogous for the right one. Since all arrows are algebra morphisms, it suffices to show equality on the generators. We have
    \[
        p^d_\ast \circ \mathcal{F}(\tilde{x}) = p^d_\ast(x_0) = d \, x_0 = \mathcal{F}_d(d \, \tilde{x}) = \mathcal{F}_d \circ \tilde{p}^d_\ast(\tilde{x}). 
    \]
    Next, let $\alpha \in \Z/N\Z$. We have
    \[
        \mathcal{F}_d \circ \tilde{p}^d_\ast(\tilde{x}_\alpha) =  \begin{cases}
            \mathcal{F}_d\left( d \tilde{x}_{\nu_d^{-1}(\alpha)} \right) & \text{if } \alpha \in d \Z/N\Z \\
            0 & \text{otherwise}
        \end{cases} =  \begin{cases}
            \displaystyle d \sum_{m=1}^{\frac{N}{d}} \zeta_N^{-m d (\nu_d \circ \iota_d)^{-1}(\alpha)} x_{\zeta_N^{dm}} & \text{if } \alpha \in d \Z/N\Z \\
            0 & \text{otherwise}
        \end{cases}
    \]
    On the other hand, we have
    \begin{align*}
        p^d_\ast \circ \mathcal{F}(\tilde{x}_\alpha) & = \mbox{\footnotesize$\displaystyle\sum_{m=1}^N \zeta_N^{-m \iota^{-1}(\alpha)} x_{\zeta_N^{d m}}$} = \mbox{\footnotesize$\displaystyle\sum_{m=1}^{\frac{N}{d}} \zeta_N^{-m \iota^{-1}(\alpha)} x_{\zeta_N^{d m}} + \sum_{m=\frac{N}{d}+1}^{2\frac{N}{d}} \zeta_N^{-m \iota^{-1}(\alpha)} x_{\zeta_N^{d m}} + \cdots + \sum_{m=(d-1)\frac{N}{d}+1}^N \zeta_N^{-m \iota^{-1}(\alpha)} x_{\zeta_N^{d m}}$} \\
        & = \sum_{j=1}^d \sum_{m_j=1}^{\frac{N}{d}} \zeta_N^{-\left(m_j+(j-1)\frac{N}{d}\right) \iota^{-1}(\alpha)} x_{\zeta_N^{d \left(m_j+(j-1)\frac{N}{d}\right)}} = \sum_{j=1}^d \sum_{m_j=1}^{\frac{N}{d}} \zeta_N^{-\left(m_j+(j-1)\frac{N}{d}\right) \iota^{-1}(\alpha)} x_{\zeta_N^{d m_j}} \\
        & = \sum_{m=1}^{\frac{N}{d}} \sum_{j=1}^d \zeta_N^{-\left(m + (j-1)\frac{N}{d}\right) \iota^{-1}(\alpha)} x_{\zeta_N^{d m}} = \sum_{m=1}^{\frac{N}{d}} \zeta_N^{-m \iota^{-1}(\alpha)} \sum_{j=1}^d \zeta_N^{-(j-1)\frac{N}{d} \iota^{-1}(\alpha)} x_{\zeta_N^{d m}},
    \end{align*}
    where the third equality comes by setting the change of variables $m_j = m - (j-1)\frac{N}{d}$ for any $j \in \llbracket 1, d\rrbracket$ and the fourth one from the factorization by $x_{\zeta_N^{dm}}$, for any $m \in \llbracket 1, \frac{N}{d}\rrbracket$. The summation in the middle is the sum of terms of a geometric sequence with first term $1$ and common ratio $\zeta_N^{-\frac{N}{d}\iota^{-1}(\alpha)}$. In order to compute this summation, we need to consider two cases:
    \begin{caselist}
        \item Assume that $\alpha \in d\Z/N\Z$. 
        {Therefore, there exists a unique $a \in \llbracket 1, \frac{N}{d}\rrbracket$ such that $\alpha = \nu_d \circ \iota_d(a)$. It follows from the commutativity of diagram \eqref{diag:nu_d} that
        \[
            \alpha = \nu_d \circ \iota_d(a) = d \iota(a) = \iota(d a). 
        \]
        This implies that
        $\iota^{-1}(\alpha) = d (\nu_d \circ \iota_d)^{-1}(\alpha)$,
                which implies that
        $\zeta_N^{-\frac{N}{d}\iota^{-1}(\alpha)} = \zeta_N^{-\frac{N}{d} d (\nu_d \circ \iota_d)^{-1}(\alpha)} = 1$.
        }
        We then obtain
        \[
            p^d_\ast \circ \mathcal{F}(\tilde{x}_\alpha) = d \sum_{m=1}^{\frac{N}{d}} \zeta_N^{-m d (\nu_d \circ \iota_d)^{-1}(\alpha)} x_{\zeta_N^{d m}}.
        \]
        \item Assume that $\alpha \notin d \Z/N\Z$. Therefore, $\zeta_N^{-\frac{N}{d}\iota^{-1}(\alpha)} \neq 1$ and then
        \[
            \sum_{j=1}^d \zeta_N^{-(j-1)\frac{N}{d} \iota^{-1}(\alpha)} = 0,
        \]
        which implies that 
        \[
            p^d_\ast \circ \mathcal{F}(\tilde{x}_\alpha) = 0.
        \]
    \end{caselist}
    Thus proving the equality $p^d_\ast \circ \mathcal{F}(\tilde{x}_\alpha) = \mathcal{F}_d \circ \tilde{p}^d_\ast(\tilde{x}_\alpha)$ for any $\alpha \in \Z/N\Z$. 
\end{proof}

\begin{rem}
    Let $d$ be a divisor of $N$ and consider positive integers $r, k_1, \dots, k_r$ with $k_1 \neq 1$ and elements $\alpha_1, \dots, \alpha_r \in d\Z/N\Z$. 
    Let $\beta_1, \dots, \beta_r$ be the unique elements in $\Z/\frac{N}{d}\Z$ such that {$\nu_d(\beta_i) = \alpha_i$}.
    Then we have
    $$
        d^{k_1 + \cdots + k_r}\zeta_{(\alpha_1, \dots, \alpha_r)}^{\bmod N}{(k_1, \dots, k_r)}  = \zeta_{(\beta_1, \dots, \beta_r)}^{\bmod \frac{N}{d}}{(k_1, \dots, k_r)}.
    $$
    Consequently, we obtain the distribution relation
    \begin{equation*}
        d^{k_1 + \cdots + k_r}\zeta_{(\alpha_1, \dots, \alpha_r)}^{\bmod N}{(k_1, \dots, k_r)}
        = \sum_{\substack{d \alpha_i^\prime = \alpha_i \\ 1 \leq i \leq r}} \zeta_{(\alpha_1^\prime, \dots, \alpha^\prime_r)}^{\bmod N}{(k_1, \dots, k_r)}.
      \end{equation*}
\end{rem}

\begin{defn}
    Let $\dmrd_0^{[N]}$ be the $\Q$-Lie subalgebra of $\dmr_0^{[N]}$ defined by
    \[
        \dmrd_0^{[N]} := \Big\{ \widetilde{\psi} \in \dmr_0^{[N]} \mid \forall d | N, \ \tilde{p}^d_\ast(\widetilde{\psi}) = \tilde{i}_d^\ast(\widetilde{\psi}) + (\widetilde{\psi} \mid \tilde{x}_0) \ \sum_{\alpha \in \Z / N \Z} \tilde{x}_\alpha \Big\}.
    \]
\end{defn}

\begin{cor}
    The map $\mathcal{F}$ induces a $\Q(\mu_N)$-Lie algebra isomorphism
    \[
        \left(\Q(\mu_N) \ \hat{\otimes}_\Q \ \dmrd_0^{[N]}, \ \langle-,-\rangle^{\widetilde{}}\right) \xrightarrow{\simeq} \left(\Q(\mu_N) \ \hat{\otimes}_\Q \ \dmrd_0^{\mu_N},\  \langle-,-\rangle\right).
    \]
\end{cor}
\begin{proof}
    This follows by applying Theorem \ref{main theorem}, Lemma \ref{lem:pi_tildepi} to the defining identities of both Lie algebras and from the identity
    \[
        \mathcal{F}_d \Big((\widetilde{\psi} \mid \tilde{x}_0) \ \sum_{\alpha \in \Z / N \Z} \tilde{x}_\alpha\Big) = \sum_{\zeta \in \mu_{\frac{N}{d}}} (\mathcal{F}_d(\widetilde{\psi}) \mid x_\zeta) x_1,
    \]
    for any $\widetilde{\psi} \in \Q\langle\langle\widetilde{X}\rangle\rangle$
\end{proof}

\begin{rem}
    It can be shown that Theorem \ref{Galois Descent theorem} extend to the pair $(\dmrd_0^{[N]}, \dmrd_0^{\mu_N})$, preserving the $G_N$-invariance characterization property established for $(\mathfrak{dmr}_0^{[N]}, \mathfrak{dmr}_0^{\mu_N})$.
\end{rem}
%%%%%%%%%%%%%%%%%%%%%%%%%%%%%%%%%%%%%%%%%%%%%%%%%%%%%%%%%%%
\bibliographystyle{amsplain}
\bibliography{FY_main}

\end{document}